\documentclass[12pt]{article}
\usepackage{amsthm}
\usepackage{amssymb}
\usepackage{amsfonts}
\usepackage{mathrsfs}
\usepackage[all,cmtip]{xy}
\usepackage{tikz}
\usepackage{tikz-cd}
\usepackage{multirow}
\usepackage{url}
\usepackage{graphicx}
\usepackage{epstopdf}
\usepackage{mathtools}
\usepackage{leftidx}
\usepackage{enumerate}
\usepackage{enumitem}
\usepackage{wrapfig}

\usetikzlibrary{cd}%quiver

\usepackage{indentfirst}
\numberwithin{equation}{section}

\DeclareFontFamily{U}{cal}{}
\DeclareFontShape{U}{cal}{m}{n}{<->cmsy10}{}

\DeclareSymbolFont{rcal}{U}{cal}{m}{n}
\DeclareSymbolFontAlphabet{\mathcal}{rcal}

\newtheorem{Def}{Definition}[section]
\newtheorem{Prop}[Def]{Proposition}
\newtheorem{Theo}[Def]{Theorem}
\newtheorem{Lem}[Def]{Lemma}
\newtheorem{Koro}[Def]{Corollary}
\theoremstyle{definition}
\newtheorem{Rem}[Def]{Remark}

\newcommand{\bsm}{\begin{smallmatrix}}
\newcommand{\esm}{\end{smallmatrix}}

\newcommand{\add}{{\rm add}}
\newcommand{\gd}{{\rm gl.dim }}

\def\ed{\mathop{\rm ext.dim}\nolimits}
\def\oed{\mathop{\rm \Omega\text{-}ext.dim}\nolimits}

\newcommand{\rad}{{\rm rad}}

\newcommand{\pd}{{\rm pd}}
\newcommand{\id}{{\rm id}}

\newcommand{\M}{{\mathcal M}}

\newcommand{\cpx}[1]{#1^{\bullet}}
\newcommand{\D}[1]{{\mathscr D}(#1)}

\newcommand{\Db}[1]{{\mathscr D}^b(#1)}
\newcommand{\C}[1]{{\mathscr C}(#1)}

\newcommand{\Cb}[1]{{\mathscr C}^b(#1)}
\newcommand{\K}[1]{{\mathscr K}(#1)}

\newcommand{\Kb}[1]{{\mathscr K}^b(#1)}
\newcommand{\modcat}{\ensuremath{\text{{\rm -mod}}}}

\newcommand{\Modcat}{\ensuremath{\text{{\rm -Mod}}}}

\newcommand{\stmodcat}[1]{#1\mbox{{\rm -{\underline{mod}}}}}

\newcommand{\pmodcat}[1]{#1\mbox{{\rm -proj}}}
\newcommand{\imodcat}[1]{#1\mbox{{\rm -inj}}}

\newcommand{\opp}{^{\rm op}}

\newcommand{\Hom}{{\rm Hom}}

\newcommand{\lra}{\longrightarrow}
\newcommand{\lraf}[1]{\stackrel{#1}{\lra}}
\newcommand{\ra}{\rightarrow}

\newcommand{\Kfb}[1]{{\mathscr K}^{-,b}(#1)}
\newcommand{\Kpb}[1]{{\mathscr K}^{+,b}(#1)}

\title{ \bf   Extension dimensions under singular equivalences
and recollements
\footnotetext{
2020 Mathematics Subject Classification: 16E05, 16G10, 16E10, 18G80.}\\
\footnotetext{
Keywords: Extension dimension; Singular equivalence of Morita type with level; Recollement.}
\footnotetext{Email addresses: zhangjb@ahu.edu.cn, zhengjunling@cjlu.edu.cn}
}

\author {Jinbi Zhang, Junling Zheng
\thanks{Corresponding author} \\
{\it \scriptsize  School of Mathematical Sciences, Anhui University, Hefei, 230601, Anhui Province, PR China}\\
{\it \scriptsize  Department of Mathematics, China Jiliang University, Hangzhou, 310018, Zhejiang Province, PR China}
}
\date{}
\begin{document}

\maketitle
\begin{abstract}
The extension dimensions of an Artin algebra give a reasonable way of measuring how far an algebra is from being representation-finite.  In this paper we mainly study  extension dimensions linked by recollements of derived module categories and singular equivalences of Morita type with level, and establish a series of new  inequalities and relationships among their extension dimensions. 
\end{abstract}

\section{Introduction}
The dimension of a triangulated category, introduced by Rouquier in \cite{r06,r08}, is an invariant that measures how quickly the triangulated category can be constructed from a single object. This invariant provides valuable insights into the representation theory of Artin algebras (see \cite{han09, kk06, Ma2024,opp09, r06}) and plays a key role in calculating the representation dimension of Artin algebras (see \cite{opp09, r06}). 
Analogous to the dimension of triangulated category, the extension dimension of abelian categories was introduced by Beligiannis in \cite{b08}. For an Artin algebra $A$, we define $\ed(A)$ as the extension dimension of the category of all finitely generated left $A$-modules. Beligiannis demonstrated that $\ed(A) = 0$ if and only if $A$ is representation-finite. Thus, the extension dimension provides a meaningful way to measure how far an algebra is from being representation-finite. 
While many upper bounds for the extension dimension of a given Artin algebra have been established (see \cite{b08, zh22, zheng2020}), determining its precise value remains a challenging task. 
To gain a deeper understanding of the extension dimensions for arbitrary Artin algebras, it is reasonable to investigate the relationships between the  extension dimensions of different Artin algebras that are connected in certain nice ways.

A useful type of natural linkages among algebras is recollements of derived module categories, introduced by Beilinson, Bernstein and Deligne in \cite{BBD}. A recollement between the derived module categories  of algebras $A$, $B$ and $C$
\begin{align*}
\xymatrixcolsep{4pc}\xymatrix{
\D{B\Modcat} \ar[r]|{i_*=i_!} &\D{A\Modcat} \ar@<-2ex>[l]|{i^*} \ar@<2ex>[l]|{i^!} \ar[r]|{j^!=j^*}  &\D{C\Modcat} \ar@<-2ex>[l]|{j_!} \ar@<2ex>[l]|{j_{*}}
}
\end{align*} provides an analogue of exact sequences for derived module categories. This framework has proven to be highly effective for understanding the relationships among algebras. 
For example, it has been applied in the study of global and finitistic dimensions \cite{akly17,cx17,happel93}, Hochschild (co)homology \cite{han14,Keller98,kn09}, K-theory \cite{akly17,cx12,sch06}, self-injective dimensions \cite{qin18}, and other homological properties such as the Igusa-Todorov properties \cite{ww22}, the Han conjecture \cite{wxzz24} and the Auslander-Reiten conjecture \cite{chqw23}. 
In this paper, we explore the behavior of extension dimensions under recollements of derived categories of Artin algebras.

To establish concrete bounds for extension dimensions, we adopt the concept of homological widths for complexes, introduced by Chen and Xi \cite{cx17}. 
Broadly speaking, the homological width of a bounded complex of projective modules measures, up to homotopy equivalence, how large the minimal interval is over which its non-zero terms are distributed. Particularly, if a module has finite projective dimension, then its homological width corresponds exactly to the projective dimension. 
For a complex $\cpx{X}$, we write $w(\cpx{X})$ for the homological width of $\cpx{X}$.

For an Artin algebra $A$, we denote by $A\Modcat$ the category of all left $A$-modules, by $A\modcat$ the category of all finitely generated left $A$-modules. The concept of extension dimension for an algebra can be extended to subcategories of its module category. We denote by $\ed(\mathcal{A})$ the extension dimension of the subcategory $\mathcal{A}$ of all finitely generated left $A$-modules. Then the extension dimension $\ed(A)$ refers to the extension dimension of $A\modcat$.
we define $\Omega^{n}(A\modcat)$ as the full subcategory of $A\modcat$ consisting of all $n$-th syzygies $A$-modules and $\Omega^{\infty}(A\modcat)$ as the full subcategory of $A\modcat$ consisting of all $\infty$-th syzygies $A$-modules. Since the extension dimension of the $i$-th syzygy module categories for $i\in \mathbb{N}$ stabilizes for sufficiently large  $i$ (see Definition \ref{ed-def}), we define this constant as the $\Omega$-extension dimension of $A$, denoted by $\oed(A)$.

The following result establishes key bounds for extension dimensions under recollements:

\begin{Theo}\label{main-thm-ed-rec}
{\rm (Theorem \ref{thm-ext-rec})}
Let $A$, $B$ and $C$ be three Artin algebras.
Suppose that there is a recollement among the derived categories
$\D{A\Modcat}$, $\D{B\Modcat}$ and $\D{C\Modcat}$$:$
\begin{align}\label{main-thm-ed-rec-f}
\xymatrixcolsep{4pc}\xymatrix{
\D{B\Modcat} \ar[r]|{i_*=i_!} &\D{A\Modcat} \ar@<-2ex>[l]|{i^*} \ar@<2ex>[l]|{i^!} \ar[r]|{j^!=j^*}  &\D{C\Modcat}. \ar@<-2ex>[l]|{j_!} \ar@<2ex>[l]|{j_{*}}
}
\end{align}

{\rm (1)} Suppose that the recollement {\rm (\ref{main-thm-ed-rec-f})} extends one step downwards. Then 

{\rm (i)} $\ed(B)\le \ed(A)+w(i_{*}(B))$.

{\rm (ii)} $\ed(C)\le \ed(A)+w(j^{*}(A))$.

{\rm (iii)} $ \ed(A)\le 2\ed(B)+\ed(C)+\max\{w(i_*(B),w(j^*(A))\}+2.$

{\rm (2)} Suppose that the recollement {\rm (\ref{main-thm-ed-rec-f})} extends one step downwards and extends one step upwards. Then 

{\rm (i)} $\ed(A)\le \ed(B)+\ed(C)+\max\{w(i_*(B),w(j^*(A))\}+1.$

{\rm (ii)} $\max\{ \oed(B),\oed(C)\}\le \oed(A)\le \oed(B)+\oed(C)+1.$

{\rm (iii)} 
%$\max\{ \ed(\Omega^{\infty}(B\modcat)),\ed(\Omega^{\infty}(C\modcat))\}\le \ed(\Omega^{\infty}(A\modcat))\le \ed(\Omega^{\infty}(B\modcat))+\ed(\Omega^{\infty}(C\modcat))+1.$
$\max\{ \ed(\Omega^{\infty}(B\modcat)),\ed(\Omega^{\infty}(C\modcat))\}\le \ed(\Omega^{\infty}(A\modcat)),\text{ and}$
$$ \ed(\Omega^{\infty}(A\modcat))\le \ed(\Omega^{\infty}(B\modcat))+\ed(\Omega^{\infty}(C\modcat))+1.$$

{\rm (iv)} If, in addition, $\gd(C)<\infty$, then $$\oed(B)=\oed(A)\text{ and } \ed(\Omega^{\infty}(B\modcat))=\ed(\Omega^{\infty}(A\modcat)).$$

{\rm (v)} If, in addition, $\gd(B)<\infty$, then $$\oed(C)=\oed(A)\text{ and } \ed(\Omega^{\infty}(C\modcat))=\ed(\Omega^{\infty}(A\modcat)).$$
\end{Theo}

Another natural linkages between algebras is through singular equivalences of Morita type with level, a concept introduced by Wang in \cite{Wang15}. This generalizes stable equivalences of Morita type, as introduced by Broué \cite{Bro90}, and singular equivalences of Morita type in the sense of Chen and Sun \cite{CS12} (see also \cite{ZZ13}). To provide concrete examples of this idea, Chen, Liu, and Wang in \cite{CLW23} demonstrated that a tensor functor with a bimodule defines a singular equivalence of Morita type with level under certain conditions. Additionally, Dalezios in \cite{Dal21} showed that for certain Gorenstein algebras, a singular equivalence induced by tensoring with a complex of bimodules always leads to a singular equivalence of Morita type with level. Building on this, Qin in \cite{qin22} provided a complex version of Chen-Liu-Wang's work, thereby generalizing Dalezios' result to arbitrary algebras, not limited to Gorenstein algebras. Recently, in \cite{qxzz24}, the authors proved that bounded extension can induce new examples of singular equivalences of Morita type with level.
This type of equivalence has profound implications in understanding homological properties, such as the finitistic dimension conjecture (\cite{Wang15}) and Keller's conjecture for singular Hochschild cohomology \cite{CLW20}. In this paper, we explore the behavior of extension dimensions under singular equivalences of Morita type with level.

\begin{Theo}\label{main-thm-level-ed}
{\rm (Theorem \ref{thm-level-ed})}
Let $A$ and $B$ be two finite dimensional algebras.
If $A$ and $B$ are singularly equivalent of Morita type with level $l$, then

{\rm (1)} $\mid\ed(A)-\ed(B)\mid\le l.$

{\rm (2)} $\oed(A)=\oed(B)$.

{\rm (3)} $\ed(\Omega^{\infty}(A\modcat))=\ed(\Omega^{\infty}(B\modcat)).$
\end{Theo}

This paper is organized as follows: In Section 2, we introduce necessary definitions and recall some key results. Section 3 investigates the behavior of extension dimensions under recollements of derived categories for Artin algebras. Section 4 examines the behavior of extension dimensions under singular equivalences of Morita type with level.

\section{Preliminaries}

In this section, we shall fix some notations, and recall some definitions.

Throughout this paper, $k$ is an arbitrary but fixed commutative Artin ring. Unless stated otherwise, all algebras are Artin $k$-algebras with unit, and all modules are finitely generated unitary left modules; all categories will be $k$-categories and all functors are $k$-functors.

Let $A$ be an Artin algebra and $\ell\ell(A)$ stand for the Loewy length of $A$. We denote by $A\Modcat$ the category of all left $A$-modules, by $A\modcat$ the category of all finitely generated left $A$-modules, by $\pmodcat{A}$ the category of all finitely generated projective $A$-modules, and by $\imodcat{A}$ the category of all finitely generated projective $A$-modules. All subcategories of $A\Modcat$ are full, additive and closed under isomorphisms.
For a class $\mathcal{C}$ of $A$-modules, we write $\add(\mathcal{C})$ for the smallest full subcategory of $A\modcat$ containing $\mathcal{C}$ and closed under finite direct sums and direct summands.
When $\mathcal{C}$ consists of only one module $C$, we write $\add(C)$ for $\add(\mathcal{C})$.
In particular, $\add({_A}A)=\pmodcat{A}$.
Let $M$ be an $A$-module.
If $f:P\ra M$ is the projective cover of $M$ with $P$ projective, then the kernel of $f$ is called the \emph{syzygy} of $M$, denoted by $\Omega(M)$.
Dually, if $g:M\ra I$ is the injective envelope of $M$ with $I$ injective, then the cokernel of $g$ is called the \emph{cosyzygy} of $M$, denoted by $\Omega^{-1}(M)$. Additionally, let $\Omega^0$ be the identity functor in $A\modcat$ and $\Omega^1:=\Omega$. Inductively, for any $n\ge 2$, define $\Omega^n(M):=\Omega^1(\Omega^{n-1}(M))$ and $\Omega^{-n}(M):=\Omega^{-1}(\Omega^{-n+1}(M))$.
We denoted by $\pd(_AM)$ and $\id(_AM)$ the projective and injective dimension, respectively.

Let $A^{\rm {op}}$ be the opposite algebra of $A$, and $A^e = A\otimes _k A^{\opp}$ be the enveloping algebra of $A$. We identify
$A$-$A$-bimodules with left $A^e$-modules. Let $D:=\Hom_k(-,E(k/\rad(k)))$ the usual duality from $A\modcat$ to $A^{\rm {op}}\modcat$, where $\rad(k)$ denotes the radical of $k$ and $E(k/\rad(k))$ denotes the injective envelope of $k/\rad(k)$. The duality $\Hom_A(-, A)$ from $\pmodcat{A}$ to $\pmodcat{A\opp}$ is denoted by $^*$, namely for each projective $A$-module $P$, the projective $A\opp$-module $\Hom_A(P, A)$ is written as $P^*$. We write $\nu_A$ for the Nakayama functor $D\Hom_A(-,A): \pmodcat{A}\ra \imodcat{A}$.

\medskip
Let $\cal C$ be an additive category. For two morphisms $f:X\rightarrow Y$ and $g:Y\rightarrow Z$ in $\cal C$, their composition is denoted by $fg$, which is a morphism from $X$ to $Z$. But for two functors $F:\mathcal{C}\ra \mathcal{D}$ and $G:\mathcal{D}\ra\mathcal{E}$ of categories, their composition is written as $GF$.

Suppose $\mathcal{C}\subseteq A\Modcat$. A (co)chain complex $\cpx{X}=(X^i, d_X^i)$ over $\mathcal{C}$ is a sequence of objects $X^i$ in
$\cal C$ with morphisms $d_{\cpx{X}}^{i}:X^i\ra X^{i+1}$ such that $d_{\cpx{X}}^{i}d_{\cpx{X}}^{i+1}=0$ for all $i \in {\mathbb Z}$.
A (co)chain map $f$ from $\cpx{X}=(X^i, d_X^i)$ to $\cpx{Y}=(Y^i, d_Y^i)$, is a set of maps $\cpx{f}=\{f^i:X^i\ra Y^i\mid i\in \mathbb{Z}\}$ such that $f^id^i_{\cpx{Y}}=d^i_{\cpx{X}}f^{i+1}$. 
Given a chain map $\cpx{f}:\cpx{X}\to\cpx{Y}$, its mapping cone is denoted by ${\rm Con}(\cpx{f})$. For an integer $n$, the $n$-th cohomology of $\cpx{X}$ is denoted by $H^n(\cpx{X})$. Let $\sup{(\cpx{X})}$ and $\inf{(\cpx{X})}$ be the supremum and minimum of indices $i\in \mathbb{Z}$ such that $H^i(\cpx{X})\ne 0$, respectively. If $\cpx{X}$ is acyclic, that is, $H^i(\cpx{X})=0$ for all $i\in\mathbb{Z}$, then we understand that $\sup(\cpx{X})=-\infty$ and $\inf(\cpx{X})=+\infty$. If $\cpx{X}$ is not acyclic, then  $\inf(\cpx{X})\le \sup(\cpx{X})$, and $H^n(\cpx{X})=0$ if $\sup{(\cpx{X})}$ is an integer and $n>\sup{(\cpx{X})}$ or if $\inf{(\cpx{X})}$ is an integer and $n<\inf{(\cpx{X})}$.
For a complex $\cpx{X}$, the complex $\cpx{X}[1]$ is obtained from $\cpx{X}$ by shifting $\cpx{X}$ to the left by one degree, and the complex $\cpx{X}[-1]$ is obtained from $\cpx{X}$ by shifting $\cpx{X}$ to the right by one degree. For $n\in \mathbb{Z}$, we denoted by $\cpx{X}_{\le n}$ (respectively, $\cpx{X}_{\ge n}$) the brutal trucated complex which is obtained from the given complex $\cpx{X}$ by replacing each $X^i$ with $0$ for $i>n$ (respectively, $i<n$).
A complex $\cpx{X}=(X^i, d_X^i)$
is called \textit{bounded above} (respectively, \textit{bounded below}) if $X^{i} = 0$ for all but finitely many positive (respectively, negative) integers $i$. A complex $\cpx{X}$ is called \textit{bounded} if it is both bounded above and bounded below, equivalently,
$X^{i} = 0$ for all but finitely many $i$. A complex $\cpx{X}$ is called \textit{cohomologically bounded} if all but finitely many cohomologies of $\cpx{X}$ are zero.

We denote by $\C{\mathcal C}$  the category of all complexes over $\cal C$ with chain map, by $\K{\mathcal C}$ the homotopy category of complexes over $\mathcal{C}$, by $\mathscr{K}^{-,b}(\mathcal C)$ the homotopy category of bounded above complexes over $\mathcal{C}$, by $\mathscr{K}^{+,b}(\mathcal C)$ the homotopy category of bounded below complexes over $\mathcal{C}$, and by $\mathscr{K}^{b}(\mathcal C)$ the homotopy category of bounded complexes over $\mathcal{C}$. Let $I$ be a subset of $\mathbb{Z}$. We denoted by $\mathscr{K}^{I}(\mathcal C)$ the subcategory of $\mathscr{K}(\mathcal C)$ consisting of complexes with the $i$-th component is $0$ for each $i\notin I$. For instance, $\mathscr{K}^{\{0\}}(\mathcal C)=\mathcal{C}$.
If $\cal C$ is an abelian category, then let $\D{\cal C}$ be the derived category of complexes over $\cal C$, $\Db{\mathcal C}$ be the full subcategory of $\D{\cal C}$ consisting of cohomologically bounded complexes over $\mathcal{C}$, and $\mathscr{D}^{I}(\mathcal C)$ be the subcategory of $\D{\mathcal C}$ consisting of complexes with the $i$-th cohomology is $0$ for any $i\in I$. For instance, $\mathscr{D}^{\{0\}}(\mathcal C)=\mathcal{C}$.

Let $A$ be an Artin algebra and  $\cpx{X}\in\D{A\Modcat}$. Suppose that $\cpx{X}$ is isomorphic in $\Db{A\modcat}$ to a bounded complex
$\cpx{Q}\in\Cb{\pmodcat{A}}$. Then the \textit{homological width} of $\cpx{X}$ is defined as:
$$w(\cpx{X})=\min\left\{
\alpha_{\cpx{P}}-\beta_{\cpx{P}}\,\geq 0\, \bigg |
\begin{array}{ll}
\cpx{P}\simeq \cpx{X} \;\;\mbox{in}\;\;\Db{A\modcat}\;\;\mbox{for}\;\;
\cpx{P}\in\Cb{\pmodcat{A}} \\
\;\;\mbox{with}\;\; P^i=0\;\;\mbox{for}\;\; i<\beta_{\cpx{P}}\;\;\mbox{or}\;\; i>\alpha_{\cpx{P}}
\end{array}
\right\}.
$$
Dually, if $\cpx{X}$ is isomorphic in $\Db{A\modcat}$ to a bounded complex
$\cpx{J}\in\Cb{\imodcat{A}}$, then the \textit{homological cowidth} of $\cpx{X}$ is defined as:
$$
cw(\cpx{Y})=\min\left\{
\alpha_{\cpx{I}}-\beta_{\cpx{I}}\,\geq 0\, \bigg  |
\begin{array}{ll}
\cpx{I}\simeq \cpx{Y} \;\;\mbox{in}\;\;\D{R}\;\;\mbox{for}\;\;
\cpx{I}\in\Cb{\imodcat{A}} \\
\;\;\mbox{with}\;\; I^i=0\;\;\mbox{for}\;\; i<\beta_{\cpx{I}}\;\;\mbox{or}\;\; i>\alpha_{\cpx{I}}
\end{array}
\right\}.
$$

In this paper, all functors between triangulated categories are assumed to be triangle functors. For convenience, we do not distinguish $\Kb{\pmodcat{A}}$, $\Kb{\imodcat{A}}$ and $\Db{A\modcat}$ from their essential images under the canonical full embeddings into $\D{A\Modcat}$. Furthermore, we always identify $A\Modcat$ with the full subcategory of $\D{A\Modcat}$ consisting of all stalk complexes concentrated on degree $0$.

\subsection{Extension dimensions}
In this subsection, we shall recall the definition and some results of the extension dimensions (see \cite{b08,zheng2020}).

Let $A$ be an Artin algebra. We denote by $A\modcat$ the category of all finitely generated left $A$-modules. For a class $\M$ of $A$-modules, we denote by $\add(\M)$ the smallest full subcategory of $A\modcat$ containing $\M$ and closed under finite direct sums and direct summands. When $\M$ consists of only one object $M$, we write $\add(M)$ for $\add(\M)$.
Let $\M_1,\M_2,\cdots,\M_n$ be subcategories of $A\modcat$. Define
\begin{align*}
\M_1\bullet \M_2
:&=\add(\{X\in A\modcat \mid \mbox{there exists an exact sequence } 0\lra M_1\lra  X \lra M_2\lra 0\\
&\qquad\qquad\qquad\qquad\quad\mbox{in } A\modcat \mbox{ with }M_1 \in \M_1\mbox{ and }M_2 \in \M_2\})\\
&=\{X\in A\modcat \mid \mbox{there exists an exact sequence } 0\lra M_1\lra  X\oplus X' \lra M_2\lra 0\\
&\qquad\qquad\qquad\quad\mbox{ in } A\modcat \mbox{ for some } A\mbox{-module } X' \mbox{ with }M_1 \in \M_1\mbox{ and }M_2 \in \M_2\}.
\end{align*}
The operation $\bullet$ is associative. Inductively, define
$$\M_{1}\bullet  \M_{2}\bullet \dots \bullet\M_{n}:=\add(\{X\in A\modcat \mid \mbox{there exists an exact sequence }
0\ra M_1\ra  X \ra M_2\ra 0$$
$$\mbox{in } A\modcat \mbox{ with }M_1 \in \M_1\mbox{ and }M_2 \in \M_{2}\bullet \dots \bullet\M_{n}\}).$$
Thus $X\in \M_{1}\bullet  \M_{2}\bullet \dots \bullet\M_{n}$ if and only if there exist the following exact sequences
\begin{equation*}
\begin{cases}
\xymatrix@C=1.5em@R=0.1em{
0\ar[r]& M_1 \ar[r]& X\oplus X_1' \ar[r]& X_2 \ar[r]& 0,\\
0\ar[r]& M_2 \ar[r]& X_2\oplus X_2' \ar[r]& X_3 \ar[r]& 0,\\
&&\vdots&&\\
0\ar[r]& M_{n-1} \ar[r]& X_{n-1}\oplus X_{n-1}' \ar[r]& X_n \ar[r]& 0,}
\end{cases}
\end{equation*}
for some $A$-modules $X_i'$ such that $M_i\in \M_i$ and $M_{i+1}\in \M_{i+1}\bullet \M_{i+2}\bullet \cdots \bullet \M_n$ for $1\le i\le n-1$.

For a subcategory $\M$ of $A\modcat$, set $[\M]_{0}:=\{0\}$, $[\M]_{1}:=\add(\M)$,  $[\M]_{n}=[\M]_1\bullet [\M]_{n-1}$ for any $n\ge 2$. If $M\in A\modcat$, we write $[M]_{n}$ instead of $[\{M\}]_{n}$.

\begin{Lem}\label{lem-ex-plus}
{\rm (\cite[Corollary 2.3(1)]{zheng2020})} For each $M,N\in A\modcat$ and $m,n\in \mathbb{Z}$, we have $$[M]_m\bullet [N]_n\subseteq [M\oplus N]_{m+n}.$$
\end{Lem}

\begin{Lem}\label{lem-shift-module} 
Let $A$ be an Artin algebra.

{\rm (1) (see \cite[Lemma 4.6]{zheng22})} Given the exact sequence $0 \ra X\ra Y\ra Z\ra 0$
in $A\modcat$, we can get the following exact sequences
$$0 \lra \Omega^{i+1}(Z) \lra \Omega^{i}(X)\oplus P_{i}\lra \Omega^{i}(Y)\lra 0,$$
$$0 \lra \Omega^{i}(X) \lra \Omega^{i}(Y)\oplus Q_{i}\lra \Omega^{i}(Z)\lra 0,$$
for some projective modules $P_{i}$ and $Q_i$ in $A\modcat$, where $i\ge 0.$

{\rm (2)} Given the exact sequence $0 \ra X\ra Y\ra Z\ra 0$
in $A\modcat$, we can get the following exact sequences
$$0 \lra \Omega^{-i}(Y) \lra \Omega^{-i}(Z)\oplus I^{i}\lra \Omega^{-i-1}(X)\lra 0,$$
$$0 \lra \Omega^{-i}(X) \lra \Omega^{-i}(Y)\oplus J^{i}\lra \Omega^{-i}(Z)\lra 0$$
for some injective modules $I^{i}$ and $J^{i}$ in $A\modcat$, where $i\ge 0.$
\end{Lem}

\begin{Lem}\label{lem-wrd} 
{\rm(\cite[Lemma 3.5]{zh22})}
Let $A$ be an Artin algebra.

{\rm (1)} Let $0\ra X\ra M^0\ra M^1 \cdots \ra M^s\ra 0$ be an exact sequence in $A\modcat$. Then
$$X\in [\Omega^{n}(M^s)]_1 \bullet [\Omega^{n-1}(M^{s-1})]_1
\bullet\cdots\bullet [\Omega^{1}(M^{1})]_1
\bullet [M^0]_1 
\subseteq [\bigoplus_{i=0}^s\Omega^{i}(M^i)]_{s+1}.$$

{\rm (2)} Let $0\ra M_{s}\ra M_{s-1}\ra \cdots\ra M_0\ra X\ra 0$ be an exact sequence in $A\modcat$. Then
$$X\in [M_0]_1\bullet[\Omega^{-1}(M_{1})]_1
\bullet\cdots\bullet
[\Omega^{-s+1}(M_{s-1})]_1\bullet [\Omega^{-s}(M_s)]_1
\subseteq [\bigoplus_{i=0}^s\Omega^{-i}(M_i)]_{s+1}.$$
\end{Lem}

\begin{Lem}\label{lem-ext-pi} 
Let $A$ be an Artin algebra and $n,s\ge 0$. Let $M$ and $X$ be  $A$-modules.

$(1)$ If $X\in [M]_{n+1}$, then $\Omega^{s}(X) \in [\Omega^{s}(M)]_{n+1}$ and $\Omega^{-s}(X) \in [\Omega^{-s}(M)]_{n+1}$.

$(2)$ If $\Omega^{-s}(X) \in [M]_{n+1}$, then $X \in [\Omega^{s}(M)\oplus\big(\bigoplus_{i=0}^{s-1}\Omega^{i}(D(A))\big)]_{n+s+1}$.

$(3)$ If $\Omega^{s}(X) \in [M]_{n+1}$, then $X \in [\Omega^{-s}(M)\oplus\big(\bigoplus_{i=0}^{s-1}\Omega^{-i}(A)\big)]_{n+s+1}$.
\end{Lem}
\begin{proof}

(1) It follows from $X\in [M]_{n+1}$ that there are short exact sequences
\begin{equation*}
\begin{cases}
\xymatrix@C=1.5em@R=0.1em{
0\ar[r]& M_0\ar[r] & X\oplus X_{0}'\ar[r]& X_{1}\ar[r]&0,\\
0\ar[r]& M_1\ar[r] & X_{1}\oplus X_{1}'\ar[r]& X_{2}\ar[r]&0,\\
&&\vdots&&\\
0\ar[r]& M_{n-1}\ar[r] & X_{n-1}\oplus X_{n-1}'\ar[r]& X_{n}\ar[r]&0.
}
\end{cases}
\end{equation*}
in $A\modcat$ for some $A$-modules $X_i'$ such that $M_i\in [M]_1$ and $X_{i+1}\in [M]_{n-i}$ for each $0\le i\le n-1$.
By Lemma \ref{lem-shift-module}(1), there are short exact sequences
\begin{equation*}
\begin{cases}
\xymatrix@C=1.5em@R=0.1em{
0\ar[r]& \Omega^{s}(M_0)\ar[r] & \Omega^{s}(X)\oplus \Omega^{s}(X_{0}')\oplus P_0\ar[r]& \Omega^{s}(X_{1})\ar[r]&0,\\
0\ar[r]& \Omega^{s}(M_1)\ar[r] & \Omega^{s}(X_{1})\oplus \Omega^{s}(X_{1}')\oplus P_1\ar[r]& \Omega^{s}(X_{2})\ar[r]&0,\\
&&\vdots&&\\
0\ar[r]& \Omega^{s}(M_{n-1})\ar[r] & \Omega^{s}(X_{n-1})\oplus \Omega^{s}(X_{n-1}')\oplus P_{n-1}\ar[r]& \Omega^{s}(X_{n})\ar[r]&0.
}
\end{cases}
\end{equation*}
for some projective module $P_{i}$ in $A\modcat$. Thus $\Omega^{s}(X) \in [\Omega^{s}(M)]_{n+1}$. Similarly, we get $\Omega^{-s}(X) \in [\Omega^{-s}(M)]_{n+1}$.

(2) We have a long exact sequence $0\ra X\ra I^0\ra I^1\ra \cdots \ra I^{s-1}\ra \Omega^{-s}(X) \ra 0$ with $I^i\in \add(D(A))$. Then
\begin{align*}
X&\in  [\Omega^{s}( \Omega^{-s}(X))]_1 \bullet [\Omega^{s-1}(I^{s-1})]_1
\bullet\cdots\bullet [\Omega^{1}(I^{1})]_1
\bullet [I^0]_1 \qquad \text{(by Lemma \ref{lem-wrd}(1))}\\
&\subseteq  [\Omega^{s}(M)]_{n+1} \bullet [\Omega^{s-1}(I^{s-1})]_1
\bullet\cdots\bullet [\Omega^{1}(I^{1})]_1
\bullet [I^0]_1 \qquad \text{(by (1))}\\
&\subseteq [\Omega^{s}(M)\oplus\big(\bigoplus_{i=0}^{s-1}\Omega^{i}(D(A))\big)]_{n+s+1}.
\end{align*}

By Lemma \ref{lem-wrd}(1), we get
$$X\in  [\Omega^{s}( \Omega^{-s}(X))]_1 \bullet [\Omega^{s-1}(I^{s-1})]_1
\bullet\cdots\bullet [\Omega^{1}(I^{1})]_1
\bullet [I^0]_1 \subseteq [\bigoplus_{i=0}^n\Omega^{i}(I^i)]_{s+1}.$$

(3) It can be proved by a similar argument as in  (2).
\end{proof}

Let $A$ be an Artin algebra and $A\modcat$ be the category of finitely generated left $A$-modules. An $A$-module $K$ is called an $n$\textit{-th syzygy module} ($n>0$) if there is an exact sequence of $A$-modules: $0\ra X\ra P_0\ra P_1\ra \cdots \ra P_{n-1}\ra Y\ra 0$ for some $A$-module $Y$ with $P_i$ projective; $X$ is called an $\infty$\textit{-th syzygy module} if there is an exact sequence of $A$-modules: $0\ra K\ra P_0\ra P_1\ra P_2 \cdots $ with $P_i$ projective. Denoted by $\Omega^{n}(A\modcat)$
the full subcategory of $A\modcat$
consisting of all $n$-th syzygies $A$-modules and by $\Omega^{\infty}(A\modcat)$
the full subcategory of $A\modcat$
consisting of all $\infty$-th syzygies $A$-modules.
For convenience, set $\Omega^{0}(A\modcat):=A\modcat$. Clearly, for nonnagative $n$, we have
$$\Omega^{n}(A\modcat)=\{K\oplus P\in A\modcat \mid K\cong\Omega^{n}(M) \text{ for some }A\text{-module } M\text{ and } P\in\pmodcat{A}\}.$$

\begin{Def}\label{ed-def}
Let $A$ be an Artin algebra and $\mathcal{A}$ be a subcategory of $A\modcat$.
The extension dimension of $\mathcal{A}$, denoted by $\ed(\mathcal{A})$, is defined to be
$$\ed(\mathcal{A})=\inf\{m\geq  0\mid \mathcal{A}\subseteq [M]_{m+1} \text{ for some }  M\in A\modcat\}.$$
The extension dimension of the algebra $A$, denoted by $\ed(A)$, is defined to be the extension dimension of the category $A\modcat$. That is,
\begin{align*}
 \ed(A):= & \ed(A\modcat )\\
  =& \inf\{m\geq  0 \mid A\modcat\subseteq [M]_{m+1} \text{ for some } M\in A\modcat\} \\
  =& \inf\{m\geq  0\mid A\modcat=[M]_{m+1}\mbox{ for some } M\in A\modcat\}.
\end{align*}
\end{Def}
Clearly, if $\mathcal{A}\subseteq \mathcal{B}\subseteq A\modcat$, then
$\ed(\mathcal{A}) \le \ed(\mathcal{B}).$ In particular, $$\ed(\Omega^{\infty}(A\modcat))\le \ed(\Omega^{p+1}(A\modcat))\le \ed(\Omega^p(A\modcat)) \text{ for any } p\ge 0.$$ By \cite[Example 1.6]{b08}, $\ed(A)\le \ell\ell(A)-1$, where $\ell\ell(A)$ stands for the Loewy length of $A$. This implies $\ed(\Omega^{i}(A\modcat))=\ed(\Omega^{i+1}(A\modcat))$ for sufficiently large $i$. Then the \emph{$\Omega$-extension dimension} of algebra $A$ is defined to be
$$\oed(A):=\min\{\ed(\Omega^{i}(A\modcat))\mid i\ge 0\}.$$
A natural question is whether $\ed(\Omega^{\infty}(A\modcat))$ is equal to $\oed(A)$. A counterexample is provided in \cite[Example 2.19]{z2024}.

\subsection{Syzygies in derived categories}\label{sec-syzygy}
In this subsection, we recall some basic facts on syzygy complexes and cosyzygy complexes for later use, as detailed in reference \cite{ai07,wei17}.

Let $A$ be an Artin algebra.
A homomorphism $f: X\ra Y$ of $A$-modules is called a \emph{radical homomorphism} if, for any indecomposable module $Z$ and homomorphisms $h: Z\ra X$ and $g: Y\ra Z$, the composition $hfg$ is not an isomorphism.
For a complex $(X^i, d_{\cpx{X}}^i)$ over $A\modcat$, if all $d_{\cpx{X}}^i$ are radical homomorphisms, then it is called a \emph{radical complex}, which has the following properties.
\begin{Lem}\label{lem-radical}
{\rm (\cite[pp. 112-113]{hx10})}
Let $A$ be an Artin algebra.

$(1)$ Every complex over $A\modcat$ is isomorphic to a radical complex in $\K{A\modcat}$.

$(2)$ Two radical complexes $\cpx{X}$ and $\cpx{Y}$ are isomorphic in $\K{A\modcat}$ if and only if they are isomorphic in $\C{A\modcat}$.
\end{Lem}

Recall that $\Db{A\modcat}$ is equivalent to $\Kfb{\pmodcat{A}}$ as triangulated categories. For a complex $\cpx{X}\in \Db{A\modcat}$, a \textit{minimal projective resolution} of $\cpx{X}$ is a radical complex $\cpx{P}\in\Kfb{\pmodcat{A}}$ such that $\cpx{P}\simeq \cpx{X}$ in $\D{A\modcat}$. In case $X$ is an $A$-module, $\cpx{P}$ is just the usual minimal projective resolution of $X$. 

\begin{Def}{\rm (\cite[Section 1.3]{ai07})}
{\rm 
Let $\cpx{X}\in \Db{A\modcat}$ and $n\in \mathbb{Z}$. Let $\cpx{P}$ be a minimal projective resolution of $\cpx{X}$. We say that a complex in $\Db{A\modcat}$ is an $n$\textit{-th syzygy} of $\cpx{X}$ provided that it is isomorphic to $\cpx{P}_{\le -n}[-n]$ in $\D{A\modcat}$. In this case, the $n$-th syzygy of $\cpx{X}$ is denoted by $\Omega_{\mathscr{D}}^n(\cpx{X}_{\cpx{P}})$, or simply by $\Omega_{\mathscr{D}}^n(\cpx{X})$ if there is no danger of confusion.}
\end{Def}
\begin{Rem}
Given a complex $\cpx{X}\in \Db{A\modcat}$, the definition of the $n$-th syzygy of $\cpx{X}$  we adopt here differs slightly from Wei's definition(\cite{wei17}). We take the projective resolution of $\cpx{X}$ to be minimal. According to Lemma \ref{lem-radical}, we know that the $n$-th syzygy of $\cpx{X}$ is unique up to isomorphism.
\end{Rem}

The following list some basic properties of syzygy complexes in \cite{wei17}.
\begin{Lem}\label{lem-prop-omega}
{\rm (\cite[Lemma 3.3]{wei17})}
Let $\cpx{X}, \cpx{Y}\in\Db{A\modcat}$, and $m,n,s,t$ be integers. Let $\cpx{P}$ be the minimal projective resolutions of $\cpx{X}$. Then

{\rm (1)} $\Omega_{\mathscr{D}}^n(\cpx{X})\in \mathscr{D}^{(-\infty,0]}(A\modcat)$ and $\Hom_{\Db{A\modcat}}(Q,\Omega_{\mathscr{D}}^n(\cpx{X})[i])=0$ for any projective $A$-module $Q$ and any integer $i>0$.

{\rm (2)}
If $\cpx{X}\in \mathscr{D}^{[s,t]}(A\modcat)$ for some integers $s\le t$, then $\Omega_{\mathscr{D}}^n(\cpx{X})\in \mathscr{D}^{[0,0]}(A\modcat),$ $\mathscr{D}^{[s+n,0]}(A\modcat),$ $\mathscr{D}^{[s+n,t+n]}(A\modcat)$ for cases $n\ge -s$, $-t\le n\le -s$, $n\le -t$, respectively. In particular, $\Omega_{\mathscr{D}}^{-s}(\cpx{X})$ is isomorphic to an $A$-module and $\cpx{X}\simeq \Omega_{\mathscr{D}}^{-t}(\cpx{X})[-t]$.

{\rm (3)} $\Omega_{\mathscr{D}}^{n+m}(\cpx{X}[m])\simeq \Omega_{\mathscr{D}}^n(\cpx{X})$.

{\rm (4)} $\Omega_{\mathscr{D}}^{n+m}(\cpx{X})\simeq \Omega_{\mathscr{D}}^m(\Omega_{\mathscr{D}}^n(\cpx{X}))$ for $m\ge 0$.

%{\rm (5)} $\Omega_{\mathscr{D}}^n(\cpx{X})\oplus Q$ is also an $n$-th syzygy of $\cpx{X}$, where $Q$ is a projective module.
%
{\rm (5)}
$\cpx{X}\in \Kb{\pmodcat{A}}$ if and only if any/some syzygy of $\cpx{X}$ is also in $\Kb{\pmodcat{A}}$.

{\rm (6)} $\Omega_{\mathscr{D}}^n(\cpx{X}\oplus \cpx{Y})\simeq \Omega_{\mathscr{D}}^n(\cpx{X})\oplus \Omega_{\mathscr{D}}^n(\cpx{Y})$.
\end{Lem}

\begin{Lem}\label{lem-omega-shift}
{\rm (\cite[Proposition 3.8]{wei17})} Let $\cpx{X}\ra \cpx{Y}\ra \cpx{Z}\ra \cpx{X}[1]$ be a triangle in $\Db{A\modcat}$, and $n$ be an integer. Then there is a trianle $\Omega_{\mathscr{D}}^n(\cpx{X})\ra \Omega_{\mathscr{D}}^n(\cpx{Y})\oplus P\ra \Omega_{\mathscr{D}}^n(\cpx{Z})\ra \Omega_{\mathscr{D}}^n(\cpx{X})[1]$ for some projective $A$-module $P$.
\end{Lem}

Recall that the dual notation of syzygy complexes, that is , cosyzygy complexes. Note that $\Db{A\modcat}$ is equivalent to $\Kpb{\imodcat{A}}$ as triangulated categories. For a complex $\cpx{X}\in \Db{A\modcat}$, a injective resolution of $\cpx{X}$ is a complex $\cpx{I}\in\Kpb{\imodcat{A}}$ such that $\cpx{I}\simeq \cpx{X}$ in $\D{A\modcat}$. In case $X$ is an $A$-module, $\cpx{I}$ is just a usual injective resolution of $X$.

\begin{Def}{\rm (\cite[Definition 3.1$'$]{wei17})}
{\rm 
Let $\cpx{X}\in \Db{A\modcat}$ and $n\in \mathbb{Z}$. Let $\cpx{I}$ be a injective resolution of $\cpx{X}$. We say that a complex in $\Db{A\modcat}$ is an $n$\textit{-th cosyzygy} of $\cpx{X}$ provied that it is isomorphic to $\cpx{I}_{\ge n}[n]$ in $\D{A\modcat}$. In this case, the $n$-th cosyzygy of $\cpx{X}$ is denoted by $\Omega^{\mathscr{D}}_n(\cpx{X}_{\cpx{I}})$, or simply by $\Omega^{\mathscr{D}}_n(\cpx{X})$ if there is no danger of confusion.
}\end{Def}

The following list some basic properties of cosyzygy complexes in \cite{wei17}.

\begin{Lem}\label{lem-prop-omega-co}
{\rm (\cite[Lemma 3.3$'$]{wei17})}
Let $\cpx{X}, \cpx{Y}\in\Db{A\modcat}$, and $m,n,s,t$ be integers. Let $\cpx{I}$ be a injective resolutions of $\cpx{X}$. Then

{\rm (1)} $\Omega^{\mathscr{D}}_n(\cpx{X})\in \mathscr{D}^{[0,\infty)}(A\modcat)$ and $\Hom_{\Db{A\modcat}}(\Omega^{\mathscr{D}}_n(\cpx{X}),J[i])=0$ for any injective $A$-module $J$ and any integer $i>0$.

{\rm (2)}
If $\cpx{X}\in \mathscr{D}^{[s,t]}(A\modcat)$ for some integers $s\le t$, then $\Omega^{\mathscr{D}}_n(\cpx{X})\in \mathscr{D}^{[0,0]}(A\modcat),$
$\mathscr{D}^{[0,t-n]}(A\modcat),$
$\mathscr{D}^{[s-n,t-n]}(A\modcat)$
for cases $n\ge t$, $s\le n\le t$, $n\le s$, respectively. In particular, $\Omega^{\mathscr{D}}_{t}(\cpx{X})$ is isomorphic to an $A$-module and $\cpx{X}\simeq \Omega_{\mathscr{D}}^{s}(\cpx{X})[s]$.

{\rm (3)} $\Omega^{\mathscr{D}}_{n+m}(\cpx{X}[m])\simeq \Omega^{\mathscr{D}}_n(\cpx{X})$.

{\rm (4)} $\Omega^{\mathscr{D}}_{n+m}(\cpx{X})\simeq \Omega^{\mathscr{D}}_m(\Omega^{\mathscr{D}}_n(\cpx{X}))$ for $m\ge 0$.

{\rm (5)}
$\cpx{X}\in \Kb{\imodcat{A}}$ if and only if any/some cosyzygy of $\cpx{X}$ is also in $\Kb{\imodcat{A}}$.

{\rm (6)} $\Omega^{\mathscr{D}}_n(\cpx{X}\oplus \cpx{Y})\simeq \Omega^{\mathscr{D}}_n(\cpx{X})\oplus \Omega^{\mathscr{D}}_n(\cpx{Y})$.
\end{Lem}

\begin{Lem}\label{lem-omega-coshift}
{\rm (\cite[Proposition 3.8$'$]{wei17})} Let $\cpx{X}\ra \cpx{Y}\ra \cpx{Z}\ra \cpx{X}[1]$ be a triangle in $\Db{A\modcat}$, and $n$ be an integer. Then there is a trianle $\Omega^{\mathscr{D}}_n(\cpx{X})\ra \Omega^{\mathscr{D}}_n(\cpx{Y})\oplus I\ra \Omega^{\mathscr{D}}_n(\cpx{Z})\ra \Omega^{\mathscr{D}}_n(\cpx{X})[1]$ for some injective $A$-module $I$.
\end{Lem}

\subsection{Recollements}

In this subsection, we recall the notation of recollements introduced by Beilinson, Bernstein and Deligne in \cite{BBD}, which is a very useful tool for representation theory and algebraic geometry.

\begin{Def}{\rm
Let $\mathscr{D}$, $\mathscr{D}'$ and $\mathscr{D}''$ be triangulated categories with shift functors denoted by $[1]$.

{\rm (1)}
{\rm (\cite{BBD})} A \emph{recollement} of $\mathscr{D}$ by $\mathscr{D}'$ and $\mathscr{D}''$ is a diagram of six triangle functors
\begin{align}\label{diag:rec}
\xymatrixcolsep{4pc}\xymatrix{\mathscr{D}' \ar[r]|{i_*=i_!} &\mathscr{D} \ar@<-2ex>[l]|{i^*} \ar@<2ex>[l]|{i^!} \ar[r]|{j^!=j^*}  &\mathscr{D}'', \ar@<-2ex>[l]|{j_!} \ar@<2ex>[l]|{j_{*}}
}
\end{align}
which satisfies

{\rm (R1)} $(i^\ast,i_\ast)$,\,$(i_!,i^!)$,\,$(j_!,j^!)$ ,\,$(j^\ast,j_\ast)$
are adjoint pairs;

{\rm (R2)}
$i_*,~j_*,~j_!$ are fully faithful$;$

{\rm (R3)}
$j^*\circ i_*=0$ $($thus $i^*\circ j_!=0$ and $i^!\circ j_*=0$$);$

{\rm (R4)}
for any object $X$ in $\mathscr{D}$, there are two triangles in $\mathscr{D}$ induced by counit and unit adjunctions$:$
$$ \xymatrix@R=0.5pc{i_*i^!(X)\ar[r] &X \ar[r] & j_*j^*(X) \ar[r]& i_*i^!(X)[1]\\
j_!j^*(X) \ar[r] &X \ar[r] & i_*i^*(X) \ar[r]& j_!j^*(X)[1]} $$
The upper two rows
\begin{align*}
\xymatrixcolsep{4pc}\xymatrix{\mathscr{D}' \ar[r]|{i_*} &\mathscr{D} \ar@<-2ex>[l]|{i^*}  \ar[r]|{j^*}  &\mathscr{D}'' \ar@<-2ex>[l]|{j_!}
}
\end{align*}
is said to be a \emph{left recollemet} of $\mathscr{D}$ by $\mathscr{D}'$ and $\mathscr{D}''$ if the four functors $i^*$, $i_*$, $j_!$ and $j^!$ satisfy the conditions in {\rm (R1)-(R4)} involving them.  A \emph{right recollement} is defined similarly via the lower two rows.

{\rm (2)}
{\rm (\cite{akly17})}
A \emph{ladder} of $\mathscr{D}$ by $\mathscr{D}'$ and $\mathscr{D}''$ is a finite or infinite diagram with triangle functors
\vspace{-10pt}$${\setlength{\unitlength}{0.7pt}
\begin{picture}(200,170)
\put(0,70){$\xymatrix@!=3pc{\mathscr{D}'
\ar@<+2.5ex>[r]|{i_{n-1}}\ar@<-2.5ex>[r]|{i_{n+1}} &\mathscr{D}
\ar[l]|{j_n} \ar@<-5.0ex>[l]|{j_{n-2}} \ar@<+5.0ex>[l]|{j_{n+2}}
\ar@<+2.5ex>[r]|{j_{n-1}} \ar@<-2.5ex>[r]|{j_{n+1}} &
\mathscr{D}''\ar[l]|{i_n} \ar@<-5.0ex>[l]|{i_{n-2}}
\ar@<+5.0ex>[l]|{i_{n+2}}}$}
\put(52.5,10){$\vdots$}
\put(137.5,10){$\vdots$}
\put(52.5,130){$\vdots$}
\put(137.5,130){$\vdots$}
\end{picture}}$$
such that any three consecutive rows form a recollement. The rows are labelled by a subset of $\mathbb{Z}$ and multiple occurence of the same recollement is allowed. The \emph{height} of a ladder is the number of recollements
contained in it $($counted with multiplicities$).$ It is an element of $\mathbb{N}\cup \{0,\infty\}$. A recollement is viewed to be a ladder of height $1$.
}\end{Def}

Recall that a sequence of triangulated categories $\mathscr{D}'\stackrel{F}{\to}\mathscr{D}\stackrel{G}{\to}\mathscr{D}''$ is said to be a \emph{short exact sequence} (up to direct summands) if $F$ is fully faithful, $G\circ F=0$ and the induced functor $\overline{G}\colon \mathscr{D}/\mathscr{D}'\to\mathscr{D}''$ is an equivalence (up to direct summands). The following result is well-known, see \cite[1.4.4, 1.4.5, 1.4.8]{BBD} and \cite[Chapter III, Lemma 1.2 (1) and Chapter IV, Proposition 1.11]{BR07}.

\begin{Prop}
\label{prop:recollement-vs-ses}
$(1)$ The two rows of a left recollement are short exact sequences of triangulated categories.
Conversely, assume that there is a short exact sequence of triangulated categories {\rm (}possibly up to direct summands{\rm)}
$$\xymatrixcolsep{2pc}\xymatrix{\mathscr{D}' \ar[r]^{i_*} &\mathscr{D}   \ar[r]^{j^*}  &\mathscr{D}''.
}$$
Then $i_*$ has a left adjoint {\rm (}respectively, right adjoint{\rm)} if and only if $j^*$ has a left adjoint {\rm (}respectively, right adjoint{\rm)}. In this case, $i_*$ and $j^*$ together with their left adjoints {\rm (}respectively, right adjoints{\rm)} form a left recollement {\rm (}respectively, right recollement{\rm)} of $\mathscr{D}$ in terms of $\mathscr{D}'$ and $\mathscr{D}''$.

$(2)$ Assume that there is a diagram 
\begin{align*}
\xymatrixcolsep{4pc}\xymatrix{\mathscr{D}' \ar[r]|{i_*=i_!} &\mathscr{D} \ar@<-2ex>[l]|{i^*} \ar@<2ex>[l]|{i^!} \ar[r]|{j^!=j^*}  &\mathscr{D}'' \ar@<-2ex>[l]|{j_!} \ar@<2ex>[l]|{j_{*}}
}
\end{align*} satisfying the condition {\rm (R1)}. If it is a recollement, then all the three rows are short exact sequences of triangulated categories. Conversely, if any one of the three rows is a short exact sequence of triangulated categories, then the diagram is a recollement.

$(3)$ Given a recollement \eqref{diag:rec}, assume that $\mathscr{D}$, $\mathscr{D}'$ and $\mathscr{D}''$ admit small coproducts. Then both $j_!$ and $i^*$ preserve compact objects.
\end{Prop}

\begin{Def}{\rm
We say that a recollement \eqref{diag:rec} \emph{extends one step downwards}  if both $i^!$ and $j_*$ have right adjoints. In this case, let $i_{?}$ and $j^{\#}$ be the right adjoints of $i^!$ and $j_*$, respectively. Then we have the following diagram:
\begin{align}\label{ladder-2}
\xymatrixcolsep{4pc}\xymatrix{\mathscr{D}' \ar[r]|{i_*=i_!} \ar@<-4ex>[r]|{i_{?}}
&\mathscr{D} \ar@<-2ex>[l]|{i^*} \ar@<2ex>[l]|{i^!} \ar[r]|{j^!=j^*} \ar@<-4ex>[r]|{j^{\#}} &\mathscr{D}''. \ar@<-2ex>[l]|{j_!} \ar@<2ex>[l]|{j_{*}}
}
\end{align}
By Proposition~\ref{prop:recollement-vs-ses}, the lower three rows also form a recollement. This means that (\ref{ladder-2}) form a ladder of height $2$. Similarly, we have the notion of extending the recollement one step upwards.
}\end{Def}

Now, we consider the recollements of derived module categories. Let $F:\Db{A\modcat}\to \Db{B\modcat}$ be a triangle functor. We say that $F$ \emph{restricts to} $\Kb{\rm proj}$ (respectively, $\Kb{\rm inj}$) if $F$ sends $\Kb{\pmodcat{A}}$ (respectively, $\Kb{\imodcat{A}}$) to $\Kb{\pmodcat{B}}$ (respectively, $\Kb{\imodcat{B}}$). Let $F:\D{A\Modcat}\to \D{B\Modcat}$ be a triangle functor. We say that $F$ \emph{restricts to} $\Db{\rm mod}$ (respectively, $\Kb{\rm proj}$, $\Kb{\rm inj}$) in a similar sense.

\begin{Lem}
{\rm (See \cite[Lemmas 3.1 and 3.2]{ww22})}
Let $A$ and $B$ be two Artin algebras, $s\le t\in \mathbb{Z}$, and $F:\D{A}\ra \D{B}$ be a triangle functor. Then the following is true.

$(1)$ If $F(A)\in\mathscr{K}^{[s,t]}(\pmodcat{B})$, then $F(\mathscr{K}^{[u,v]}(\pmodcat{B}))\subseteq\mathscr{K}^{[u+s,v+t]}(\pmodcat{B})$ for any integers $u\le v$.

$(2)$ If $F(D(A))\in\mathscr{K}^{[s,t]}(\imodcat{B})$, then $F(\mathscr{K}^{[u,v]}(\imodcat{B}))\subseteq\mathscr{K}^{[u+s,v+t]}(\imodcat{B})$ for any integers $u\le v$.

$(3)$ If $F(A\modcat)\subseteq\mathscr{D}^{[s,t]}(B\modcat)$, then $F(\mathscr{D}^{[u,v]}(B\modcat))\subseteq\mathscr{D}^{[u+s,v+t]}(B\modcat)$ for any integers $u\le v$.
\end{Lem}

The following results are well-known, see \cite{akly17,JYZ23}.
\begin{Lem}
Let $A$ and $B$ be two Artin algebras. Suppose that there are triangle functors
\begin{align*}
\xymatrixcolsep{4pc}\xymatrix{
\D{B\Modcat} \ar@<-1ex>[r]|{G}
&\D{A\Modcat} \ar@<-1ex>[l]|{F} 
}
\end{align*}
between the unbounded derived categories of $A$ and $B$ such that $(F,G)$ is an adjoint pair. Then the following is true.

$(1)$ $F$ restricts to $\Kb{{\rm proj}}$ if and only if $G$ restricts to $\Db{{\rm mod}}$.

$(2)$ $F$ restricts to $\Db{{\rm mod}}$ if and only if $G$ restricts to $\Kb{{\rm inj}}$.
\end{Lem}

\begin{Lem}\label{lem-rec-prop}
Let $A$, $B$ and $C$ be Artin algebras.
Suppose that there is a recollement among the derived categories
$\D{A\Modcat}$, $\D{B\Modcat}$ and $\D{C\Modcat}$:
\begin{align}\label{rec-der}
\xymatrixcolsep{4pc}\xymatrix{
\D{B\Modcat} \ar[r]|{i_*=i_!} &\D{A\Modcat} \ar@<-2ex>[l]|{i^*} \ar@<2ex>[l]|{i^!} \ar[r]|{j^!=j^*}  &\D{C\Modcat}. \ar@<-2ex>[l]|{j_!} \ar@<2ex>[l]|{j_{*}}
}
\end{align}
Then $i^{*}$ and $j_{!}$ restrict to $\mathscr{K}^b({\rm proj})${\rm ;} $i_{*}$ and $j^{!}$ restrict to $\mathscr{D}^b({\rm mod})$. Moreover, the following hold true.

{\rm (1)} The following conditions are equivalent$:$

{\rm (i)} the recollement {\rm (\ref{rec-der})} extends one step downwards$;$

{\rm (ii)} $i_{*}$ restricts to $\mathscr{K}^b({\rm proj})$$;$

{\rm (iii)} $i_{*}(B)\in\Kb{\pmodcat{A}}$$;$

{\rm (iv)} $j^{*}$ restricts to $\mathscr{K}^b({\rm proj})$$;$

{\rm (v)} $j^{*}(A)$ restricts to $\mathscr{K}^b({\pmodcat{C}})$$;$

{\rm (vi)} $i^!$ restricts to $\mathscr{D}^b({\rm mod})$$;$

{\rm (vii)} $j_*$ restricts to $\mathscr{D}^b({\rm mod})$$;$

{\rm (viii)} $i^!$ has a right adjoint$;$

{\rm (ix)} $j_*$ has a right adjoint.

In this case, the recollement {\rm (\ref{rec-der})} restricts to a left recollement and a right recollement
\begin{align*}
\xymatrixcolsep{4pc}\xymatrix{
\Kb{\pmodcat{B}} \ar[r]|{i_*=i_!} &\Kb{\pmodcat{A}} \ar@<-2ex>[l]|{i^*}  \ar[r]|{j^!=j^*}  &\Kb{\pmodcat{C}}, \ar@<-2ex>[l]|{j_!}
}
\end{align*}
\begin{align*}
\xymatrixcolsep{4pc}\xymatrix{
\Db{B\modcat} \ar[r]|{i_*=i_!} &\Db{A\modcat}  \ar@<2ex>[l]|{i^!} \ar[r]|{j^!=j^*}  &\Db{C\modcat}. \ar@<2ex>[l]|{j_{*}}
}
\end{align*}

{\rm (2)} The following conditions are equivalent$:$

{\rm (i)} the recollement {\rm (\ref{rec-der})} extends one step upwards$;$

{\rm (ii)} $i_{*}$ restricts to $\mathscr{K}^b({\rm inj})$$;$

{\rm (iii)} $i_{*}(D(B))\in \Kb{\imodcat{A}}$$;$

{\rm (iv)} $j^{*}$ restricts to $\mathscr{K}^b({\rm inj})$$;$

{\rm (v)} $j^{*}(D(A))\in \Kb{\imodcat{C}}$$;$

{\rm (vi)} $i^*$ restricts to $\mathscr{D}^b({\rm mod})$$;$

{\rm (vii)} $j_{!}$ restricts to $\mathscr{D}^b({\rm mod})$$;$

{\rm (viii)} $i^*$ has a left adjoint$;$

{\rm (ix)} $j_{!}$ has a left adjoint.

In this case, the recollement {\rm (\ref{rec-der})} restricts to a left recollement and a right recollement
\begin{align*}
\xymatrixcolsep{4pc}\xymatrix{
\Db{B\modcat} \ar[r]|{i_*=i_!} &\Db{A\modcat}  \ar@<-2ex>[l]|{i^*}  \ar[r]|{j^!=j^*}  &\Db{C\modcat}, \ar@<-2ex>[l]|{j_!}
}
\end{align*}
\begin{align*}
\xymatrixcolsep{4pc}\xymatrix{
\Kb{\imodcat{B}}  \ar[r]|{i_*=i_!} &\Kb{\imodcat{A}} \ar@<2ex>[l]|{i^!} \ar[r]|{j^!=j^*}  &\Kb{\imodcat{C}}. \ar@<2ex>[l]|{j_{*}}
}
\end{align*}
\end{Lem}

\subsection{Nonnegative functors}
In this subsection, we shall recall the definition and some results of the nonnegative functors (see \cite{hp17,hx10}).

\begin{Def}{\rm A triangle functor $F: \Db{A\modcat}\ra \Db{B\modcat}$ is called {\bf nonnegative} if $F$ satisfies the following conditions:

$(1)$ $F(X)$ is isomorphic to a complex with zero cohomology in all negative degrees for all $X\in A\modcat$; and

$(2)$ $F(P)$ is isomorphic to a complex in $\Kb{\pmodcat{B}}$ with zero terms in all negative degrees for all $P\in\pmodcat{A}$.
}
\end{Def}
%
%For a nonnegative $F: \Db{A\modcat}\ra \Db{B\modcat}$, the following lemma describes the images of modules in $A\modcat$ under $F$.
%
%\begin{Lem}\label{lem-basic-non}
%{\rm (\cite[Lemma 4.5]{hp17})}
%Let $F: \Db{A\modcat}\ra \Db{B\modcat}$ be a nonnegative functor.
%For each $X\in A\modcat$, there is a triangle
%$$\cpx{P}_X\lraf{i_X}F(X)\lraf{\pi_X} M_X\lraf{\mu_X}\cpx{P}_X[1]$$
%in $\Db{B\modcat}$ with $M_X\in B\modcat$ and $\cpx{P}_X\in\mathscr{D}^{[1,n_X]}(\pmodcat{B})$ for some $n_X>0$.
%\end{Lem}

Suppose that $F: \Db{A\modcat}\ra \Db{B\modcat}$ is a nonnegative functor. By \cite[Lemma 4.5]{hp17}, 
for each $X\in A\modcat$, we can fix a triangle  $$\cpx{Q}_X\lraf{i_X}F(X)\lraf{\pi_X} V_X\lraf{\mu_X}\cpx{Q}_X[1]$$
in $\Db{B\modcat}$ with $V_X\in B\modcat$ and $\cpx{Q}_X\in\mathscr{K}^{[1,n_X]}(\pmodcat{B})$ for some $n_X>0$. In \cite[Section 4]{hp17}, Hu and Pan construct a functor  $\overline{F}:\stmodcat{A}\ra \stmodcat{B},\; X\mapsto V_X$, which is called the \emph{stable functor} of $F$. This stable functor has the following properties.

\begin{Lem}\label{lem-non}{\rm (see \cite[Proposition 4.11]{hx10} and \cite[Proposition 4.13]{hp17})} Let $F: \Db{A\modcat}\ra \Db{B\modcat}$ be a nonnegative functor. Suppose that $$0 \lra X\lra Y \lra Z\lra 0$$ is an exact sequence in $A\modcat$. Then there is an exact sequence
$$0 \lra \overline{F}(X) \lra \overline{F}(Y)\oplus Q \lra \overline{F}(Z)\lra 0$$ in $B\modcat$ for some projective $B$-module $Q$.
\end{Lem}

\begin{Lem}\label{lem-non-func}
Let $F: \Db{A\modcat}\ra \Db{B\modcat}$ be a nonnegative functor and $n\ge 0$. 

$(1)$ If $X\in \Omega^n(A\modcat)$, then $\overline{F}(X) \in \Omega^n(B\modcat)$.

$(2)$ If $X\in \Omega^{\infty}(A\modcat)$, then $\overline{F}(X) \in \Omega^{\infty}(B\modcat)$.
\end{Lem}
\begin{proof}
(1) It follows from $X\in \Omega^n(A\modcat)$ that there is an $A$-module $Y$ such that we have the following exact sequence in $A\modcat$
$$0\lra X\lra P_{0}\lraf{f_0} P_{1}\lraf{f_1} P_{2}\lraf{f_2}\cdots \lraf{f_{n-2}} P_{n-1}\lraf{f_{n-1}} Y\lra 0$$
with $P_{i}\in\pmodcat{A}$ for each $i\in \mathbb{N}$.
Let $K_i$ be the kernel of $f_i$. Then $K_0=X$ and we have the following short exact sequences
\begin{equation*}
\begin{cases}
\xymatrix@C=1.5em@R=0.1em{
0\ar[r]& X\ar[r] & P_{0}\ar[r]& K_{1}\ar[r]&0,\\
0\ar[r]& K_{1}\ar[r] & P_{1}\ar[r]& K_{2}\ar[r]&0,\\
0\ar[r]& K_{2}\ar[r] & P_{2}\ar[r]& K_{3}\ar[r]&0,\\
&  &\vdots & &\\
0\ar[r]& K_{n-1}\ar[r] & P_{n-1}\ar[r]& Y\ar[r]&0.
}
\end{cases}
\end{equation*}
Applying the functor $\overline{F}$ to the above short exact sequences and by Lemma \ref{lem-non}(3), we can get the following short exact sequences
\begin{equation*}
\begin{cases}
\xymatrix@C=1.5em@R=0.1em{
0\ar[r]& \overline{F}(X)\ar[r]     & Q_{0}\ar[r]& \overline{F}(K_{1})\ar[r]&0,\\
0\ar[r]& \overline{F}(K_{1})\ar[r] & Q_{1}\ar[r]& \overline{F}(K_{2})\ar[r]&0,\\
0\ar[r]& \overline{F}(K_{2})\ar[r] & Q_{2}\ar[r]& \overline{F}(K_{3})\ar[r]&0,\\
&  &\vdots & &\\
0\ar[r]& \overline{F}(K_{n-1})\ar[r] & Q_{n-1}\ar[r]& \overline{F}(Y)\ar[r]&0
}
\end{cases}
\end{equation*}
where $_BQ_{i}$ is projective for each $0\le i\le n-1$.
And we get the following exact sequence
$$\xymatrix@C=1.5em@R=0.1em{
0\ar[r]&\overline{F}(X)\ar[r]
&Q_{0}\ar[r]&
Q_{1}\ar[r]&Q_{2}\ar[r]& \cdots \ar[r]& Q_{n-1}\ar[r]& \overline{F}(Y)\ar[r]&0.
}$$
Then  $\overline{F}(X)\in \Omega^{n}(B\modcat)$.

(2)  It follows from $X\in \Omega^{\infty}(A\modcat)$ that the
following exact sequence
$$0\lra X\lra P_{0}\lraf{f_0} P_{1}\lraf{f_1} P_{2}\lraf{f_2}\cdots$$
with $P_{i}\in\pmodcat{A}$ for each $i\in \mathbb{N}$.
Let $K_i$ be the kernel of $f_i$. Then $K_0=X$ and we have the following short exact sequences
\begin{equation*}
\begin{cases}
\xymatrix@C=1.5em@R=0.1em{
0\ar[r]& X\ar[r] & P_{0}\ar[r]& K_{1}\ar[r]&0,\\
0\ar[r]& K_{1}\ar[r] & P_{1}\ar[r]& K_{2}\ar[r]&0,\\
0\ar[r]& K_{2}\ar[r] & P_{2}\ar[r]& K_{3}\ar[r]&0,\\
&  &\vdots & &
}
\end{cases}
\end{equation*}
Applying the functor $\overline{F}$ to the above short exact sequences and by Lemma \ref{lem-non}(3), we can get the following short exact sequences
\begin{equation*}
\begin{cases}
\xymatrix@C=1.5em@R=0.1em{
0\ar[r]& \overline{F}(X)\ar[r]     & Q_{0}\ar[r]& \overline{F}(K_{1})\ar[r]&0,\\
0\ar[r]& \overline{F}(K_{1})\ar[r] & Q_{1}\ar[r]& \overline{F}(K_{2})\ar[r]&0,\\
0\ar[r]& \overline{F}(K_{2})\ar[r] & Q_{2}\ar[r]& \overline{F}(K_{3})\ar[r]&0,\\
&  &\vdots & &
}
\end{cases}
\end{equation*}
where $_BQ_{i}$ is projective for each $i\in\mathbb{N}$.
And we get the following exact sequence
$$\xymatrix@C=1.5em@R=0.1em{
0\ar[r]&\overline{F}(X)\ar[r]
&Q_{0}\ar[r]&
Q_{1}\ar[r]&Q_{2}\ar[r]& \cdots.
}$$
Then  $\overline{F}(X)\in \Omega^{\infty}(B\modcat)$.
\end{proof}

\section{Recollements and extension dimensions}\label{sect-ext-rec}

To investigate relationships among extension dimensions of algebras in recollements, it may be convenient to introduce the notion of \emph{homological widths of functors}. 

Let $A$ and $B$ be two Artin algebras. Suppose that $\mathscr{X}$ and $\mathscr{Y}$ are full subcategories of $\D{A\Modcat}$ and $\D{B\Modcat}$, respectively, and that  $A\modcat\subseteq \mathscr{X}$. For a given additive functor $F:\mathscr{X}\ra \mathscr{Y}$, set
$$
\inf{(F)}:=\inf\{n\in\mathbb{Z}\mid H^n(F(X))\neq 0 \mbox{ for some } X\in A\modcat\},
$$
$$
\sup{(F)}:=\sup\{n\in\mathbb{Z}\mid H^n(F(X))\neq 0 \mbox{ for some } X\in A\modcat\}.
$$
Note that $\inf(F)=+\infty$ if and only if $F(X)=0$ in $\D{B}$ for all $X\in A\modcat$ if and only if $\sup(F)=-\infty$. We define the \textit{homological width} of  $F$ to be $w(F):=\sup{(F)}-\inf{(F)}$.

\begin{Lem}\label{lem-wid}
Let $F:\D{A\Modcat}\ra \D{B\Modcat}$ be a triangle functor. Then the following statements are ture.

$(1)$ Suppose that $F$ has a left adjoint $E$. If $E(B)\in \Kb{\pmodcat{A}}$, then $$\inf(F)= -\sup(E(B)),\;\sup(F)= w(E(B))-\sup(E(B)),\;w(F)=w(E(B)).$$
%In particular, $E(B)\in\mathscr{K}^{[-\sup(F),-\inf(F)]}(\pmodcat{A})$.
If, in addition, $F(A)\in \Kb{\pmodcat{B}}$, then $$\sup(F(A))=w(E(B))-\sup(E(B))\text{ and }w(F(A))\ge w(E(B)).$$

$(2)$ Suppose that $F$ has a right adjoint $G$. If $G(D(B))\in \Kb{\imodcat{A}}$, then $$\inf(F)= -\inf(G(D(B)))-cw(G(D(B))),\;\sup(F)= -\inf(G(D(B))),\;w(F)=cw(G(D(B))).$$
%In particular, $G(B)\in\mathscr{K}^{[-\sup(F),-\inf(F)]}(\imodcat{A})$.
If, in addition, $F(D(A))\in \Kb{\imodcat{B}}$, then $$\inf(F(D(A))= -\inf(G(D(B)))-cw(G(D(B)))\text{ and }cw(F(D(A)))\ge cw(G(D(B))).$$
\end{Lem}

\begin{proof}
(1) For each $n\in \mathbb{Z}$ and $X\in A\modcat$,
\begin{align}\label{iso-ch}
H^n(F(X))\simeq\Hom_{\D{B}}(B,F(X)[n])\simeq\Hom_{\D{A}}(E(B),X[n])\simeq\Hom_{\K{A}}(E(B),X[n]).
\end{align}
Denote $t:=\sup(E(B))$ and $s:=\sup(E(B))-w(E(B))$.
Since $E(B)\in \Kb{\pmodcat{A}}$, $E(B)$ can be written as
$$
\cpx{P}:\; 0\lra P^{s}\lraf{d^{s}} P^{s+1}\lraf{d^{s+1}} P^{s+2}\cdots \lra P^{t-1}\lraf{d^{t-1}} P^{t} \lra 0
$$
with each term being projective. It follows the isomorphisms (\ref{iso-ch}) that
$$H^n(F(X))\simeq\Hom_{\K{A}}(E(B),X[n])=0$$ 
for all $n<-t$ or $n>-s$. Thus
$$\inf(F)\ge -t,\;\sup(F)\le -s,\;w(F)\le t-s=w(E(B)).$$

Denote $u:=\inf(F)$ and $v:=\sup(F)$. It follows from the isomorphisms (\ref{iso-ch})
that $\Hom_{\K{A}}(E(B),X[n])=0$ for all $n<u$ or $n>v$. This implies $E(B)\in \mathscr{K}^{[-v,-u]}(\pmodcat{A})$ and $w(E(B))\le v-u=w(F)$. Hence $w(E(B))=w(F)$.

Suppose $F(A)\in \Kb{\pmodcat{B}}$. Due to  $w(E(B))=t-s$,  $d^s$ is not a split injection and $P^t\neq 0$.  Also
$H^n(F(A))\simeq\Hom_{\K{A}}(E(B),A[n])$. Then $H^n(F(A))\neq 0$ for $n=-s$ or $n=-t$. Thus $\sup(F(A))=-s$ and $w(F(A))\ge t-s=w(E(B))$.

(2) %To calculate cohomologies of complexes, we conside the usual duality  $$D:=\Hom_k(-,J): B\modcat\lra B^{\rm {op}}\modcat,$$ where $\rad(k)$ denotes the radical of $k$ and $J$ denotes the injective envelope of $k/\rad(k)$.
For each $n\in \mathbb{Z}$ and $X\in A\modcat$,
\begin{align}\label{iso-dh}
D(H^n(F(X)))&\simeq\Hom_{k}(H^n(F(X)),J)
\simeq\Hom_{\K{k}}(F(X)[n],J)\nonumber\\
&\simeq\Hom_{\K{k}}(B\otimes_BF(X)[n],J)
\simeq\Hom_{\K{B}}(F(X)[n],D(B))\nonumber\\
&\simeq\Hom_{\D{B}}(F(X)[n],D(B))\simeq\Hom_{\D{A}}(X[n],G(D(B)))\nonumber\\
&\simeq\Hom_{\K{A}}(X[n],G(D(B))).
\end{align}
Denote $s:=\inf(D(B))$ and $t:=\inf(D(B))+cw(D(B))$.
Since $G(D(B))\in \Kb{\imodcat{A}}$, $G(D(B))$ can be written as
$$
\cpx{I}:\; 0\lra I^{s}\lraf{d^{s}} I^{s+1}\lraf{d^{s+1}} I^{s+2}\cdots \lra I^{t-1}\lraf{d^{t-1}} I^{t} \lra 0
$$
with each term being injective. It follows from (\ref{iso-dh}) that $H^n(F(X))=0$ for $n<-t$ or $n>-s$. Then $$\inf(F)\ge -t,\;\sup(F)\le -s,\;w(F)\le t-s=cw(G(D(B))).$$

Denote $u:=\inf(F)$ and $v:=\sup(F)$. It follows from (\ref{iso-dh}) that $\Hom_{\K{A}}(X[n],G(D(B)))=0$ for all $X\in A\modcat$ and $n<-u$ or $n>-v$. This implies $G(D(B))\in \mathscr{K}^{[-v,-u]}(\imodcat{A})$ and $w(F)=cw(G(D(B)))\le v-u=w(F)$. Hence $cw(G(D(B)))=w(F)$.

Suppose $F(D(A))\in \Kb{\imodcat{B}}$.  It follows from (\ref{iso-dh}) that  $$D(H^n(F(D(A))))\simeq\Hom_{\K{A}}(D(A)[n],G(D(B))).$$ Due to  $cw(G(D(B)))=t-s$,  $d^{t-1}$ is not a split surjection and $I^s\neq 0$. Thus
$H^n(F(A))\neq 0$ for $n=-s$ or $n=-t$.
Hence $\inf(F(D(A))= -t$ and $cw(G(D(B)))\ge cw(G(D(B)))$.
\end{proof}

\begin{Lem}\label{lemma-fun-sys}
Let $F:\D{A\Modcat}\ra \D{B\Modcat}$ be a triangle functor with $w(F)<\infty$ and $t\in \mathbb{Z}$. Let $X$ be an $A$-module and there exists an $A$-module $M$ such that $X\in [M]_{m+1}$ for some nonegative integer $m$. Then

$(1)$ $\Omega_{\mathscr{D}}^{p}(F(X)[t])\in [ \Omega_{\mathscr{D}}^p(F(M)[t])]_{m+1}$ for any $p\ge -\inf(F)+t$.

$(2)$ $\Omega^{\mathscr{D}}_p(F(X)[t])\in [ \Omega^{\mathscr{D}}_p(F(M)[t])]_{m+1}$ for any $p\ge \sup(F)-t$.
\end{Lem}

\begin{proof}
It follows from $X\in [M]_{m+1}$ that there are short exact sequences
\begin{equation*}
\begin{cases}
\xymatrix@C=1.5em@R=0.1em{
0\ar[r]& M_0\ar[r] & X\oplus X_{0}'\ar[r]& X_{1}\ar[r]&0,\\
0\ar[r]& M_1\ar[r] & X_{1}\oplus X_{1}'\ar[r]& X_{2}\ar[r]&0,\\
&&\vdots&&\\
0\ar[r]& M_{m-1}\ar[r] & X_{m-1}\oplus X_{m-1}'\ar[r]& X_{m}\ar[r]&0
}
\end{cases}
\end{equation*}
in $A\modcat$ for some $A$-modules $X_i'$ such that $M_i\in [M]_1$ and $X_{i+1}\in [M]_{m-i}$ for each $0\le i\le m-1$. 
Then we have triangles in $\Db{A\modcat}$:
\begin{equation}\label{tri-bb}
\begin{cases}
\xymatrix@C=1.5em@R=0.1em{
 M_0\ar[r] & X\oplus X_{0}'\ar[r]& X_{1}\ar[r]&M_0[1],\\
 M_1\ar[r] & X_{1}\oplus X_{1}'\ar[r]& X_{2}\ar[r]&M_1[1],\\
&\vdots&&\\
 M_{m-1}\ar[r] & X_{m-1}\oplus X_{m-1}'\ar[r]& X_{m}\ar[r]&M_{m-1}[1]
}
\end{cases}
\end{equation}
Applying the functor $[t]\circ F$ to the triangles (\ref{tri-bb}), we have triangles in $\Db{B\modcat}$:
\begin{equation}\label{f-t-tri}
\begin{cases}
\xymatrix@C=1.5em@R=0.1em{
 F(M_0)[t]\ar[r] & F(X)[t]\oplus F(X_{0}')[t]\ar[r]& F(X_{1})[t]\ar[r]&F(M_0)[t+1],\\
 F(M_1)[t]\ar[r] & F(X_{1})[t]\oplus F(X_{1}')[t]\ar[r]& F(X_{2})[t]\ar[r]&F(M_1)[t+1],\\
&\vdots&&\\
 F(M_{m-1})[t]\ar[r] & F(X_{m-1})[t]\oplus F(X_{m-1}')[t]\ar[r]& F(X_{m})[t]\ar[r]&F(M_{m-1})[t+1]
}
\end{cases}
\end{equation}

(1) By Lemma \ref{lem-omega-shift}, for each $p\in \mathbb{Z}$, applying $ \Omega_{\mathscr{D}}^p$ to the triangle (\ref{f-t-tri}), we have triangles in $\Db{B\modcat}$:
{\small\begin{equation}\label{tri-omega-fx}
\begin{cases}
\xymatrix@C=1.5em@R=0.1em{
 \Omega_{\mathscr{D}}^p(F(M_0)[t])\ar[r] 
& \Omega_{\mathscr{D}}^p(F(X)[t])\oplus \Omega_{\mathscr{D}}^p(F(X_{0}')[t])\oplus Q_0\ar[r]
& \Omega_{\mathscr{D}}^p(F(X_{1})[t])\ar[r]& \Omega_{\mathscr{D}}^p(F(M_0)[t])[1],\\
 \Omega_{\mathscr{D}}^p(F(M_1)[t])\ar[r] 
& \Omega_{\mathscr{D}}^p(F(X_{1})[t])\oplus \Omega_{\mathscr{D}}^p(F(X_{1}')[t])\oplus Q_1
\ar[r]& \Omega_{\mathscr{D}}^p(F(X_{2})[t])\ar[r]&\Omega_{\mathscr{D}}^p(F(M_1)[t])[1],\\
&\vdots&&\\
 \Omega_{\mathscr{D}}^p(F(M_{m-1})[t])\ar[r] & \Omega_{\mathscr{D}}^p(F(X_{m-1})[t])\oplus \Omega_{\mathscr{D}}^p(F(X_{m-1}')[t])\oplus Q_{m-1}
\ar[r]& \Omega_{\mathscr{D}}^p(F(X_{m})[t])\ar[r]&\Omega_{\mathscr{D}}^p(F(M_{m-1})[t])[1]
}
\end{cases}
\end{equation}}
for some $Q_i\in\pmodcat{B}$, $0\le i\le m-1$. 

Suppose $p\ge -\inf(F)+t$.
Notice $w(F)<\infty$. Then for each $W\in A\modcat$, $F(W)\in \mathscr{D}^{[\inf(F),\sup(F)]}(B\modcat)$ and $F(W)[t] \in \mathscr{D}^{[\inf(F)-t,\sup(F)-t]}(B\modcat).$  
By Lemma \ref{lem-prop-omega} (2), 
all term in the triangles (\ref{tri-omega-fx}) are $B$-modules. Applying the cohomological functors $H^i:=\Hom_{\Db{B\modcat}}(B,-[i])$, $i=-1,0,1$ 
to the triangles (\ref{tri-omega-fx}), we have short exact sequences
{\small\begin{equation}\label{tri-omega-fx-ex}
\begin{cases}
\xymatrix@C=1.5em@R=0.1em{
0\ar[r]& \Omega_{\mathscr{D}}^p(F(M_0)[t])\ar[r] 
& \Omega_{\mathscr{D}}^p(F(X)[t])\oplus \Omega_{\mathscr{D}}^p(F(X_{0}')[t])\oplus Q_0\ar[r]
& \Omega_{\mathscr{D}}^p(F(X_{1})[t])\ar[r]&0,\\
0\ar[r]&  \Omega_{\mathscr{D}}^p(F(M_1)[t])\ar[r] 
& \Omega_{\mathscr{D}}^p(F(X_{1})[t])\oplus \Omega_{\mathscr{D}}^p(F(X_{1}')[t])\oplus Q_1
\ar[r]& \Omega_{\mathscr{D}}^p(F(X_{2})[t])\ar[r]&0,\\
&&\vdots&\\
0\ar[r]&  \Omega_{\mathscr{D}}^p(F(M_{m-1})[t])\ar[r] & \Omega_{\mathscr{D}}^p(F(X_{m-1})[t])\oplus \Omega_{\mathscr{D}}^p(F(X_{m-1}')[t])\oplus Q_{m-1}
\ar[r]& \Omega_{\mathscr{D}}^p(F(X_{m})[t])\ar[r]&0
}
\end{cases}
\end{equation}}
Then $ \Omega_{\mathscr{D}}^p(F(X)[t])\in [ \Omega_{\mathscr{D}}^p(F(M)[t])]_{m+1}$.

(2) By Lemma \ref{lem-omega-coshift}, for each $p\in \mathbb{Z}$, applying $ \Omega^{\mathscr{D}}_p$ to the triangle (\ref{f-t-tri}), we have triangles in $\Db{B\modcat}$:
{\small\begin{equation}\label{tri-omega-fx-co}
\begin{cases}
\xymatrix@C=1.5em@R=0.1em{
\Omega^{\mathscr{D}}_p(F(M_0)[t])\ar[r] 
&\Omega^{\mathscr{D}}_p(F(X)[t])\oplus \Omega^{\mathscr{D}}_p(F(X_{0}')[t])\oplus J_0\ar[r]
&\Omega^{\mathscr{D}}_p(F(X_{1})[t])\ar[r]&\Omega^{\mathscr{D}}_p(F(M_0)[t])[1],\\
\Omega^{\mathscr{D}}_p(F(M_1)[t])\ar[r] 
&\Omega^{\mathscr{D}}_p(F(X_{1})[t])\oplus \Omega^{\mathscr{D}}_p(F(X_{1}')[t])\oplus J_1
\ar[r]&\Omega^{\mathscr{D}}_p(F(X_{2})[t])\ar[r]&\Omega^{\mathscr{D}}_p(F(M_1)[t])[1],\\
&\vdots&&\\
\Omega^{\mathscr{D}}_p(F(M_{m-1})[t])\ar[r] &\Omega^{\mathscr{D}}_p(F(X_{m-1})[t])\oplus \Omega^{\mathscr{D}}_p(F(X_{m-1}')[t])\oplus J_{m-1}
\ar[r]&\Omega^{\mathscr{D}}_p(F(X_{m})[t])\ar[r]&\Omega^{\mathscr{D}}_p(F(M_{m-1})[t])[1]
}
\end{cases}
\end{equation}}
for some $J_i\in\imodcat{B}$, $0\le i\le m-1$. 

Suppose $p\ge \sup(F)-t$.
Note that $F(W)[t] \in \mathscr{D}^{[\inf(F)-t,\sup(F)-t]}(B\modcat)$  for all $W\in A\modcat$.
By Lemma \ref{lem-prop-omega}(2), 
all term in the triangles (\ref{tri-omega-fx-co}) are $B$-modules. Applying the cohomological functors $H^i:=\Hom_{\Db{B\modcat}}(B,-[i])$, $i=-1,0,1$ 
to the triangles (\ref{tri-omega-fx-co}), we have short exact sequences
{\small\begin{equation}\label{tri-omega-fx-ex-co}
\begin{cases}
\xymatrix@C=1.5em@R=0.1em{
0\ar[r]& \Omega^{\mathscr{D}}_p(F(M_0)[t])\ar[r] 
& \Omega^{\mathscr{D}}_p(F(X)[t])\oplus \Omega^{\mathscr{D}}_p(F(X_{0}')[t])\oplus Q_0\ar[r]
& \Omega^{\mathscr{D}}_p(F(X_{1})[t])\ar[r]&0,\\
0\ar[r]& \Omega^{\mathscr{D}}_p(F(M_1)[t])\ar[r] 
& \Omega^{\mathscr{D}}_p(F(X_{1})[t])\oplus \Omega^{\mathscr{D}}_p(F(X_{1}')[t])\oplus Q_1
\ar[r]& \Omega^{\mathscr{D}}_p(F(X_{2})[t])\ar[r]&0,\\
&&\vdots&\\
0\ar[r]&  \Omega^{\mathscr{D}}_p(F(M_{m-1})[t])\ar[r] & \Omega^{\mathscr{D}}_p(F(X_{m-1})[t])\oplus \Omega^{\mathscr{D}}_p(F(X_{m-1}')[t])\oplus Q_{m-1}
\ar[r]& \Omega^{\mathscr{D}}_p(F(X_{m})[t])\ar[r]&0
}
\end{cases}
\end{equation}}
Then $\Omega^{\mathscr{D}}_p(F(X)[t])\in [ \Omega^{\mathscr{D}}_p(F(M)[t])]_{m+1}$.
\end{proof}

\begin{Lem}\label{lem-adjoint}
Let $A$ and $B$ be two Artin algebras. Suppose that there are triangle functors
\begin{align*}
\xymatrixcolsep{4pc}\xymatrix{
\D{B\Modcat} \ar@<-2ex>[r]|{G}\ar@<2ex>[r]|{E}
&\D{A\Modcat} \ar[l]|{F}
}
\end{align*}
between the unbounded derived categories of $A$ and $B$ such that $(E,F)$ and $(F,G)$ are adjoint pairs. If $E$ and $F$ restrict to $\Kb{{\rm proj}}$, then the following is true.

$(1)$ $w(F)\le w(G)<\infty$, $\sup(F)+\sup(G)=w(F(A))=w(G)$, and $\inf(F)+\inf(G)=-w(E(B))=-w(F)$.

$(2)$ If $G$ is fully faithful, then $\ed(B)\le \ed(A)+w(F(A))$.

$(3)$ If $F$ is fully faithful, then $\ed(A)\le \ed(B)+w(F(A))$.

If, in addition, $G$ restricts to $\Kb{{\rm proj}}$, then the following is true.

$(4)$ If $G$ is fully faithful, then $$\oed(B)\le \oed(A)\text{ and }\ed(\Omega^{\infty}(B\modcat))\le \ed(\Omega^{\infty}A\modcat)).$$

$(5)$ If $F$ is fully faithful, then $$\oed(A)\le \oed(B)\text{ and }\ed(\Omega^{\infty}(A\modcat))\le \ed(\Omega^{\infty}(B\modcat)).$$
\end{Lem}
\begin{proof}
(1) By Lemma \ref{lem-wid}(1), $w(F)=w(E(B))\le w(F(A))=w(G)$.
It follows from Lemma \ref{lem-wid}(1) that
$\sup(F)= w(E(B))-\sup(E(B)),\;\sup(G)= w(F(A))-\sup(F(A))\text{ and }\sup(F(A))= w(E(B))-\sup(E(B))=\sup(F).$
Thus $\sup(F)+\sup(G)=w(F(A))=w(G)$. Similarly, it follows from Lemma \ref{lem-wid}(1) that $\inf(F)=-\sup(E(B))$, $\inf(G)=-\sup(F(A))$ and $\sup(F(A))= w(E(B))-\sup(E(B))$. Then $\inf(F)+\inf(G)=-w(E(B))$.

%(3)  It follows from a similar argument as in (2).
%
(2) By Lemma \ref{lem-wid}(1), $w(G)=w(F(A))<\infty$. Then for each $Y\in B\modcat$, $G(Y)\in \mathscr{D}^{[\inf(G),\sup(G)]}(A\modcat)$. Thus there is a canonical triangle
\begin{align}\label{tri-pres}
\cpx{P} \lra G(Y)\lra X[-\inf(G)]\lra  \cpx{P}[1]
\end{align}
for some $X\in A\modcat$, where $\cpx{P}\in \mathscr{K}^{[\inf(G)+1,\sup(G)]}(\pmodcat{A})$.
Applying the functor $F$ to the triangle (\ref{tri-pres}), we get a new triangle 
\begin{align}\label{tri-pres-2}
F(\cpx{P} ) \lra FG(Y)\lra F(X[-\inf(G)])\lra  F(\cpx{P} )[1].
\end{align}
By Lemma \ref{lem-wid}(1), $F(A)\in \mathscr{K}^{[-\sup(G),-\inf(G)]}(\pmodcat{A})$. 
Then \begin{align}\label{fg-relation}
F(\cpx{P})\in F(\mathscr{K}^{[\inf(G)+1,\sup(F)]}(\pmodcat{A})) \subseteq 
\mathscr{K}^{[-\sup(G)+\inf(G)+1,-\inf(G)+\sup(G)]}(\pmodcat{A})
\end{align}
Then $\Omega_{\mathscr{D}}^p(F(\cpx{P}))=0$ for $p\ge -\inf(F)-\inf(G)$. 
By Lemma \ref{lem-omega-shift}, there is a projective $B$-module $P$ such that 
\begin{align}\label{fgy-sim-gy}
\Omega_{\mathscr{D}}^p(F(X[-\inf(G)]))\simeq \Omega_{\mathscr{D}}^p(FG(Y))\oplus P
\end{align}
for any $p\ge -\inf(F)-\inf(G)$. 

Suppose $\ed(A)=m$. By Definition \ref{ed-def}, there exists an $A$-module $M$ such that $$A\modcat=[M]_{m+1}.$$
In particular, $X\in [M]_{m+1}$. By Lemma \ref{lemma-fun-sys}(1), $$\Omega_{\mathscr{D}}^{p}(F(X)[-\inf(G)])\in [ \Omega_{\mathscr{D}}^p(F(M)[-\inf(G)])]_{m+1}$$ for any $p\ge -\inf(F)-\inf(G)$.

By (1), $-\inf(F)-\inf(G)=w(F)\le w(G)=:v$. 
Then $$\Omega_{\mathscr{D}}^{v}(F(X)[-\inf(G)])\in [ \Omega_{\mathscr{D}}^v(F(M)[-\inf(G)])]_{m+1}.$$
Since $G$ is fully faithful, $Y\simeq FG(Y)$.
By the isomorphism (\ref{fgy-sim-gy}), $$\Omega_{\mathscr{D}}^v(Y) \simeq \Omega_{\mathscr{D}}^v(FG(Y)) \in [ \Omega_{\mathscr{D}}^v(F(M)[-\inf(G)])]_{m+1}.$$
Due to $Y\in B\modcat$, $\Omega^{v}(Y)=\Omega_{\mathscr{D}}^v(Y) $.
By Lemma \ref{lem-ext-pi}(3), $$Y \in [\Omega^{-v}(\Omega_{\mathscr{D}}^v(F(M)[-\inf(G)]))\oplus\big(\bigoplus_{i=0}^{v-1}\Omega^{-i}(A)\big)]_{m+v+1}.$$
By Definition \ref{ed-def}, $\ed(B)\le m+v=\ed(A)+w(F(A))$.

(3) By Lemma \ref{lem-wid}(1), $w(F)=w(E(B))<\infty$. Then $F(X)\in \mathscr{D}^{[\inf(F),\sup(F)]}(B\modcat)$ for each $X\in A\modcat$. Thus there is a canonical triangle
\begin{align}\label{tri-ires}
Y[-\sup(F)] \lra F(X)\lra \cpx{J}\lra  Y[-\sup(F)] [1]
\end{align}
for some $Y\in B\modcat$, where $\cpx{J}\in \mathscr{K}^{[\inf(F),\sup(F)-1]}(\imodcat{B})$.
Applying the functor $G$ to the triangle (\ref{tri-ires}), we get a new triangle 
\begin{align}\label{tri-ires-2}
G(Y)[-\sup(F)] \lra GF(X)\lra G(\cpx{J})\lra  G(Y)[-\sup(F)] [1].
\end{align}
By Lemma \ref{lem-wid}(2), $G(D(B))\in \mathscr{K}^{[-\sup(F),-\inf(F)]}(\imodcat{A})$. Then \begin{align}\label{gj-relation}
G(\cpx{J})\in G(\mathscr{K}^{[\inf(F),\sup(F)-1]}(\imodcat{B}))\subseteq \mathscr{K}^{[\inf(F)-\sup(F),\sup(F)-\inf(F)-1]}(\imodcat{A}).
\end{align}
Then $\Omega^{\mathscr{D}}_p(G(\cpx{J}))=0$ for any $p\ge \sup(F)-\inf(F)$. 
By Lemma \ref{lem-omega-coshift}, there is an injective $A$-module $I$ such that 
\begin{align}\label{gfx-sim-gy}
\Omega^{\mathscr{D}}_p(GF(X))\oplus I\simeq \Omega^{\mathscr{D}}_p(G(Y)[-\sup(F)])\text{ for all } p\ge \sup(F)-\inf(F).
\end{align}

Notice $w(G)=w(F(A))<\infty$. Then $G(Y)\in \mathscr{D}^{[\inf(G),\sup(G)]}(A\modcat)$ and $$G(Y)[-\sup(F)] \in \mathscr{D}^{[\sup(F)+\inf(G),\sup(F)+\sup(G)]}(A\modcat).$$ 

Suppose $\ed(B)=n$. By Definition \ref{ed-def}, there exists a $B$-module $N$ such that $$B\modcat=[N]_{n+1}.$$
In particular, $Y\in [N]_{n+1}$. By Lemma \ref{lemma-fun-sys}(2), 
$$\Omega^{\mathscr{D}}_p(G(Y)[-\sup(F)] )\in [ \Omega^{\mathscr{D}}_p(G(N)[-\sup(F)])]_{n+1} \text{ for any } p\ge \sup(F)+\sup(F).$$
By (1), $v:=\sup(F)+\sup(G)=w(F(A))\ge w(F)=\sup(F)-\inf(F)$.
Then $$\Omega^{\mathscr{D}}_v(G(Y)[-\sup(F)])\in [\Omega^{\mathscr{D}}_v(G(N)[-\sup(F)])]_{n+1}.$$
Since $F$ is fully faithful, $X\simeq GF(X)$.
By the isomorphism (\ref{gfx-sim-gy}), $$\Omega^{\mathscr{D}}_v(X)\in [\Omega^{\mathscr{D}}_v(G(N)[-\sup(F)])]_{n+1}.$$
Due to $X\in A\modcat$, $\Omega^{-v}(X)=\Omega^{\mathscr{D}}_v(X)$.
By Lemma \ref{lem-ext-pi}(2), $$X \in [\Omega^{v}(\Omega^{\mathscr{D}}_v(G(N)[-\sup(F)])\oplus D(A))\oplus\big(\bigoplus_{i=0}^{v-1}\Omega^{i}(D(A))\big)]_{n+v+1}.$$
By Definition \ref{ed-def}, $\ed(A)\le n+v=\ed(B)+w(F(A))$.

(4) By Lemma \ref{lem-wid}(1), $\mid w(G)\mid \le w(G(B))<\infty$. Then 
$$G(B\modcat)\subseteq \mathscr{D}^{[\inf(G),\sup(G)]}(A\modcat) \text{ and }G(B)\in \mathscr{K}^{[\sup(G)-w(G(B)),\sup(G)]}(\pmodcat{A}).$$
Thus $[\sup(G)-w(G(B))]\circ G:\Db{B\modcat}\ra \Db{A\modcat}$ is nonnegative and 
$$G(B\modcat)[\sup(G)-w(G(B))]\subseteq \mathscr{D}^{[w(G(B))-w(G),w(G(B))]}(A\modcat) \subseteq  \mathscr{D}^{[0,w(G(B))]}(A\modcat).$$ 
Then for each $Y\in B\modcat$, we can fix a triangle  $$\cpx{P}_Y\lra G(Y)[\sup(G)-w(G(B))]\lra  U_Y\lra \cpx{P}_Y[1]$$
in $\Db{A\modcat}$ with $U_Y\in A\modcat$ and $\cpx{P}_Y\in\mathscr{K}^{[1,w(G(B))]}(\pmodcat{A})$.
Furthermore, we have  a triangle  
\begin{align}\label{lem-tri-omega-j}
\cpx{P}_Y[w(G(B))-\sup(G)]\lra G(Y)\lra  U_Y[w(G(B))-\sup(G)]\lra \cpx{P}_Y[w(G(B))-\sup(G)][1]
\end{align}
in $\Db{A\modcat}$ with $\cpx{P}_Y[w(G(B))-\sup(G)]\in\mathscr{K}^{[\sup(G)-w(G(B))+1,\sup(G)]}(\pmodcat{A})$.
Applying the functor $F$ to the triangles (\ref{lem-tri-omega-j}), we have a triangle in $\Db{B\modcat}$:
\begin{align}\label{lem-tri-omega-jj}
F(\cpx{P}_Y[w(G(B))-\sup(G)])\ra FG(Y)\ra  F(U_Y[w(G(B))-\sup(G)])\ra F(\cpx{P}_Y[w(G(B))-\sup(G)])[1]
\end{align}
By Lemma \ref{lem-wid}(1), $F(A)\in \mathscr{K}^{[\sup(F)-w(F(A)),\sup(F)]}(\pmodcat{B})$. Then
\begin{align*}
F(\cpx{P}_Y[w(G(B))-\sup(G)])
\in \mathscr{K}^{[\sup(F)-w(F(A))+\sup(G)-w(G(B))+1,\sup(F)+\sup(G)]}(\pmodcat{B}).
\end{align*}
By (1), $\sup(F)+\sup(G)=w(F(A))$. Then
\begin{align*}
F(\cpx{P}_Y[w(G(B))-\sup(G)])
\in \mathscr{K}^{[-w(G(B))+1,w(F(A))]}(\pmodcat{B}).
\end{align*}
Then $\Omega_{\mathscr{D}}^p(F(\cpx{P}_Y[w(G(B))-\sup(G)]))=0$ for $p\ge w(G(B))$. By Lemma \ref{lem-omega-shift} and the triangle (\ref{lem-tri-omega-jj}), there is a projective $B$-module $Q$ such that 
\begin{align}\label{lem-jj-relation-omega}
\Omega_{\mathscr{D}}^p(FG(Y))\oplus Q\simeq  \Omega_{\mathscr{D}}^p(F(U_Y[w(G(B))-\sup(G)]))\text{ for }p\ge w(G(B)).
\end{align}

Suppose $\oed(A)=m$. By the definition of the $\Omega$-extension dimension, there exist $\alpha \ge 0$ and $M\in A\modcat$ such that  $\Omega^{ \alpha}(A\modcat)\subseteq [M]_{m+1}$.  
Suppose $Y\in \Omega^{ \alpha}(B\modcat)$. By Lemma \ref{lem-non-func}(1), $U_Y\in \Omega^{\alpha}(A\modcat)$. Then $U_Y\in [M]_{m+1}$.
By Lemma \ref{lemma-fun-sys} (1),$\text{ for any } p\ge -\inf(F)+w(G(B))-\sup(G),$
\begin{align*}
 \Omega_{\mathscr{D}}^p(F(U_Y[w(G(B))-\sup(G)]))\in [ \Omega_{\mathscr{D}}^p(F(M[w(G(B))-\sup(G)]))]_{m+1}.
\end{align*}
Note that
\begin{align*}
-\inf(F)+w(G(B))-\sup(G)&=\inf(G)+w(E(B))+w(G(B))-\sup(G)\qquad\text{(by  (1))}\\
&=-w(G)+w(F)+w(G(B))\qquad\text{(by Lemma \ref{lem-wid}(1))}\\
&\le w(G(B))\qquad\text{(by (1))}.
\end{align*}
Then
\begin{align}\label{lem-iv-b-n-jw}
 \Omega_{\mathscr{D}}^p(F(U_Y[w(G(B))-\sup(G)]))\in [ \Omega_{\mathscr{D}}^p(F(M[w(G(B))-\sup(G)]))]_{m+1}
\text{ for any } p\ge w(G(B)).
\end{align}
By the relations (\ref{lem-jj-relation-omega}) and (\ref{lem-iv-b-n-jw}), we get
\begin{align}\label{lem-iv-b-n-fin-jw}
\Omega_{\mathscr{D}}^p(FG(Y))\in [ \Omega_{\mathscr{D}}^p(F(M[w(G(B))-\sup(G)]))]_{m+1}
\text{ for any } p\ge w(G(B)).
\end{align}
Since $G$ is fully faithful, $Y\simeq FG(Y)$ in $\Db{B\modcat}$. 
Then 
\begin{align}\label{lem-iv-b-n-fin-jw-fin}
\Omega_{\mathscr{D}}^p(Y)\in [ \Omega_{\mathscr{D}}^p(F(M[w(G(B))-\sup(G)]))]_{m+1}
\text{ for any } p\ge w(G(B)).
\end{align}
By the definition of the $\Omega$-extension dimension, $\oed(B)\le m=\oed(A)$.
Similarly, by Lemma \ref{lem-non-func}(2), we can obtain $\ed(\Omega^{\infty}(B\modcat))\le \ed(\Omega^{\infty}A\modcat)).$

(5)
By Lemma \ref{lem-wid}(1), $\mid w(F)\mid \le w(F(A))<\infty$. Then 
$$F(A\modcat)\subseteq \mathscr{D}^{[\inf(F),\sup(F)]}(B\modcat) \text{ and }F(A)\in \mathscr{K}^{[\sup(F)-w(F(A)),\sup(F)]}(\pmodcat{B}).$$
Thus $[\sup(F)-w(F(A))]\circ F:\Db{A\modcat}\ra \Db{B\modcat}$ is nonnegative and 
$$F(A\modcat)[\sup(F)-w(F(A))]\subseteq \mathscr{D}^{[w(F(A))-w(F),w(F(A))]}(B\modcat) \subseteq  \mathscr{D}^{[0,w(F(A))]}(B\modcat).$$ 
Then for each $X\in A\modcat$, we can fix a triangle  $$\cpx{Q}_X\lra  F(X)[\sup(F)-w(F(A))]\lra  V_X\lra \cpx{Q}_X[1]$$
in $\Db{B\modcat}$ with $V_X\in B\modcat$ and $\cpx{Q}_X\in\mathscr{K}^{[1,w(F(A))]}(\pmodcat{B})$.
Furthermore, we have  a triangle  
\begin{align}\label{LEM-tri-omega-i}
\cpx{Q}_X[w(F(A))-\sup(F)]\ra  F(X)\ra  V_X[w(F(A))-\sup(F)]\ra \cpx{Q}_X[w(F(A))-\sup(F)][1]
\end{align}
in $\Db{B\modcat}$ with $\cpx{Q}_X[w(F(A))-\sup(F)]\in\mathscr{K}^{[\sup(F)-w(F(A))+1,\sup(F)]}(\pmodcat{B})$.
Applying the functor $G$ to the triangles (\ref{LEM-tri-omega-i}), we have a triangle in $\Db{A\modcat}$:
\begin{align}\label{LEM-tri-omega-ii}
G(\cpx{Q}_X[w(F(A))-\sup(F)])\ra  GF(X)\ra G( V_X[w(F(A))-\sup(F)])\ra G( \cpx{Q}_X[w(F(A))-\sup(F)])[1].
\end{align}
By Lemma \ref{lem-wid}(1), $G(B)\in \mathscr{K}^{[\sup(G)-w(G(B)),\sup(G)]}(\pmodcat{A})$. Then
\begin{align*}
G(\cpx{Q}_X[w(F(A))-\sup(F)])\in  \mathscr{K}^{[\sup(G)-w(G(B))+\sup(F)-w(F(A))+1,\sup(F)+\sup(G)]}(\pmodcat{A}).
\end{align*}
By (1), $\sup(F)+\sup(G)=w(F(A))$. Then
$$G(\cpx{Q}_X[w(F(A))-\sup(F)])\in  \mathscr{K}^{[-w(G(B))+1,w(F(A))]}(\pmodcat{A}).$$
Then $\Omega_{\mathscr{D}}^p(G(\cpx{Q}_X[w(F(A))-\sup(F)]))=0$ for $p\ge  w(G(B))$. By Lemma \ref{lem-omega-shift} and the triangle (\ref{LEM-tri-omega-ii}), there is a projective $A$-module $P$ such that 
\begin{align}\label{LEM-ii-relation-omega}
\Omega_{\mathscr{D}}^p(GF(X))\oplus P\simeq  \Omega_{\mathscr{D}}^p(G( V_X[w(F(A))-\sup(F)]))\text{ for any }p\ge w(G(B)).
\end{align}

Suppose $\oed(B)=n$. By the definition of the $\Omega$-extension dimension, there exist $\beta\ge 0$ and $N\in B\modcat$ such that  $\Omega^{ \beta}(B\modcat)\subseteq [N]_{n+1}$.  Suppose $X\in \Omega^{\beta}(A\modcat)$. By Lemma \ref{lem-non-func}(1), $V_X\in \Omega^{\beta}(B\modcat)$. Then $V_X\in [N]_{n+1}$.
By Lemma \ref{lemma-fun-sys} (1),$\text{ for any } p\ge -\inf(G)+w(F(A))-\sup(F),$
$$\Omega_{\mathscr{D}}^{p}(G( V_X[w(F(A))-\sup(F)]))\in [ \Omega_{\mathscr{D}}^p(G( N[w(F(A))-\sup(F)]))]_{n+1}.$$
By Lemma \ref{lem-wid}(1), $\inf(G)=-\sup(F(A))$ and $\sup(F(A))=\sup(F)$. Then $-\inf(G)+w(F(A))-\sup(F)=w(F(A))$.
Thus \begin{align}\label{LEM-iv-b-n}
\Omega_{\mathscr{D}}^{p}(G( V_X[w(F(A))-\sup(F)]))\in [ \Omega_{\mathscr{D}}^p(G( N[w(F(A))-\sup(F)]))]_{n+1}\text{ for any } p\ge w(F(A)).
\end{align}
By Lemma \ref{lem-wid}(1), $w(G(B))\ge w(F(A))$.
By the relations (\ref{LEM-ii-relation-omega}) and (\ref{LEM-iv-b-n}), we get
\begin{align}\label{LEM-iv-b-n-fin}
\Omega_{\mathscr{D}}^p(GF(X)) \in [ \Omega_{\mathscr{D}}^p(G( N[w(F(A))-\sup(F)]))]_{n+1} \text{ for any } p\ge w(G(B)).
\end{align}
Since $F$ is fully faithful, $X\simeq GF(X)$ in $\Db{A\modcat}$. 
Then 
\begin{align}\label{LEM-iv-b-n-fin-FIN}
\Omega_{\mathscr{D}}^p(X) \in [ \Omega_{\mathscr{D}}^p(G( N[w(F(A))-\sup(F)]))]_{n+1} \text{ for any } p\ge w(G(B)).
\end{align}
By the definition of the $\Omega$-extension dimension, $\oed(A)\le n=\oed(B)$.
Similarly, by Lemma \ref{lem-non-func}(2), we can obtain $\ed(\Omega^{\infty}(A\modcat))\le \ed(\Omega^{\infty}B\modcat)).$
\end{proof}
An immediate consequence is the following. 
\begin{Koro}\label{cor-der}
Suppose that $F:\D{A}\ra \D{B}$ is a triangle equivalence. Then
$$\mid\ed(A)-\ed(B)\mid\le w(F(A)),$$
$$\oed(B)=\oed(A)\text{ and }\ed(\Omega^{\infty}(B\modcat))= \ed(\Omega^{\infty}A\modcat)).$$
\end{Koro}
\begin{proof}
Suppose that $G:\D{B}\to\D{A}$ is a quasi-inverse of $F$. Then $F$ and $G$ restrict to $\Kb{{\rm proj}}$, and $G$ is also a triangle equivalence. In particular, $F$ and $G$ are fully faithful.
Note that $(G,F)$ and $(F,G)$ are adjoint pairs. Then we have triangle functors
\begin{align*}
\xymatrixcolsep{4pc}\xymatrix{
\D{B\Modcat} \ar@<-2ex>[r]|{G}\ar@<2ex>[r]|{G}
&\D{A\Modcat} \ar[l]|{F}
}
\end{align*}
Thus the result follows from Lemma \ref{lem-adjoint}.
\end{proof}

\begin{Theo}\label{thm-ext-rec}
Let $A$, $B$ and $C$ be three Artin algebras.
Suppose that there is a recollement among the derived categories
$\D{A\Modcat}$, $\D{B\Modcat}$ and $\D{C\Modcat}$$:$
\begin{align}\label{thm-ext-rec-f}
\xymatrixcolsep{4pc}\xymatrix{
\D{B\Modcat} \ar[r]|{i_*=i_!} &\D{A\Modcat} \ar@<-2ex>[l]|{i^*} \ar@<2ex>[l]|{i^!} \ar[r]|{j^!=j^*}  &\D{C\Modcat}. \ar@<-2ex>[l]|{j_!} \ar@<2ex>[l]|{j_{*}}
}
\end{align}

{\rm (1)} Suppose that the recollement $(\ref{thm-ext-rec-f})$ extends one step downwards. Then 

{\rm (i)} $\ed(B)\le \ed(A)+w(i_{*}(B))$.

{\rm (ii)} $\ed(C)\le \ed(A)+w(j^{*}(A))$.

{\rm (iii)} $ \ed(A)\le 2\ed(B)+\ed(C)+\max\{w(i_*(B),w(j^*(A))\}+2.$

{\rm (2)} Suppose that the recollement $(\ref{thm-ext-rec-f})$ extends one step downwards and extends one step upwards. Then 

{\rm (i)} $\ed(A)\le \ed(B)+\ed(C)+\max\{w(i_*(B),w(j^*(A))\}+1.$

{\rm (ii)} $\max\{ \oed(B),\oed(C)\}\le \oed(A)\le \oed(B)+\oed(C)+1.$

{\rm (iii)} 
%$\max\{ \ed(\Omega^{\infty}(B\modcat)),\ed(\Omega^{\infty}(C\modcat))\}\le \ed(\Omega^{\infty}(A\modcat))\le \ed(\Omega^{\infty}(B\modcat))+\ed(\Omega^{\infty}(C\modcat))+1.$
$\max\{ \ed(\Omega^{\infty}(B\modcat)),\ed(\Omega^{\infty}(C\modcat))\}\le \ed(\Omega^{\infty}(A\modcat)),\text{ and}$
$$ \ed(\Omega^{\infty}(A\modcat))\le \ed(\Omega^{\infty}(B\modcat))+\ed(\Omega^{\infty}(C\modcat))+1.$$

{\rm (iv)} If, in addition, $\gd(C)<\infty$, then $$\oed(B)=\oed(A)\text{ and } \ed(\Omega^{\infty}(B\modcat))=\ed(\Omega^{\infty}(A\modcat)).$$

{\rm (v)} If, in addition, $\gd(B)<\infty$, then $$\oed(C)=\oed(A)\text{ and } \ed(\Omega^{\infty}(C\modcat))=\ed(\Omega^{\infty}(A\modcat)).$$
\end{Theo}

\begin{proof}
(1) Applying Lemma \ref{lem-adjoint}(2) and (3), we obtain (i) and (ii), respectively.

(iii) The given recollement (\ref{thm-ext-rec-f}) yields the following canonical triangle
\begin{align}\label{can-tri}
i_{!}i^!(X)\lraf{\eta_X} 
X\lraf{\epsilon_X}j_{*}j^{*}(X)
\lra i_{!}i^!(X)[1]
\end{align}
in $\D{A}$, where $\eta_X$ and $\epsilon_X$ stand for the counit and unit morphisms, respectively.

By Lemma \ref{lem-wid}(1), $w(i^!)=w(i_!(B))<\infty$. 
Denote $s:=\inf(i^!)$, $t:=\sup(i^!)$. Then $i^!(X)\in \mathscr{D}^{[s,t]}(B\modcat)$ and $i^!(X)$ can be written as
\begin{align*}
0\lra Y^s\lra Y^{s+1}\lra \cdots \lra Y^{t-1}\lra Y^t\lra 0.
\end{align*}
From the proof of \cite[Theorem]{han09}, we see that there is a triangle in $\Db{B\modcat}$:
\begin{align}\label{tri-kxz}
\cpx{K}\lra i^!(X)\lra \cpx{Z}\lra \cpx{K}[1],
\end{align}
where $\cpx{K}=\bigoplus_{q=s}^{t}(K^q[-q])$ and $\cpx{Z}=\bigoplus_{q=s}^{t}(Z^q[-q])$ with $K^q, Z^q\in  B\modcat$.
Applying the functor $i^!$ to the triangle, we get a new triangle
\begin{align}\label{tri-kxz-i}
\bigoplus_{q=s}^{t}(i_!(K^q)[-q])\lra i_!i^!(X)\lra \bigoplus_{q=s}^{t}(i_!(Z^q)[-q])\lra i_!( \cpx{K})[1].
\end{align}

Let $n:=\ed(B)$ and $N\in B\modcat$ with $B\modcat=[N]_{n+1}$. An analogous argument to the proof of Lemma \ref{lem-adjoint}(2) shows that for each $s\le q\le t$,
$$\Omega^{\mathscr{D}}_{\alpha}(i_!(K^q)[-q])\in [\Omega^{\mathscr{D}}_{\alpha}(i_!(N)[-q])]_{n+1},  \text{ for all }{\alpha}\ge \sup(i_!)+q.$$

Denote $v:=\sup(i_!)+t=\sup(i_!)+\sup(i^!)$. By Lemma \ref{lem-adjoint}(1), $v=w(i_!(A))$. Then
\begin{align}\label{sub-k}
 \Omega^{\mathscr{D}}_{\alpha}(\bigoplus_{q=s}^{t}(i_!(K^q)[-q]))\in [\Omega^{\mathscr{D}}_{\alpha}(\bigoplus_{q=s}^{t}(i_!(N)[-q]))]_{n+1},  \text{ for all }{\alpha}\ge v.
\end{align}
Similarly, we have 
\begin{align}\label{sub-z}
\Omega^{\mathscr{D}}_{\alpha}(\bigoplus_{q=s}^{t}(i_!(Z^q)[-q]))\in [\Omega^{\mathscr{D}}_{\alpha}(\bigoplus_{q=s}^{t}(i_!(N)[-q]))]_{n+1},  \text{ for all }{\alpha}\ge v.
\end{align}
For each $\alpha\le v$, 
applying $\Omega^{\mathscr{D}}_{\alpha}$ to the triangles (\ref{tri-kxz-i}), we have triangles in $\Db{A\modcat}$:
\begin{align}\label{tri-kxz-i-om}
\Omega^{\mathscr{D}}_{\alpha}(\bigoplus_{q=s}^{t}(i_!(K^q)[-q]))\lra
\Omega^{\mathscr{D}}_{\alpha}(i_!i^!(X))\oplus I_1
\lra \Omega^{\mathscr{D}}_{\alpha}(\bigoplus_{q=s}^{t}(i_!(Z^q)[-q]))\lra \Omega^{\mathscr{D}}_{\alpha}(i_!( \cpx{K}))[1]
\end{align}
for some injective $A$-module $I_1$ by Lemma \ref{lem-omega-shift}.
Note that all term in the triangles (\ref{tri-kxz-i-om}) are $A$-modules. Applying the cohomological functors $H^i:=\Hom_{\Db{A\modcat}}(A,-[i])$, $i=-1,0,1$ 
to the triangles (\ref{tri-kxz-i-om}), we have short exact sequences
\begin{align}\label{tri-kxz-i-om-ex}
0\lra \Omega^{\mathscr{D}}_{\alpha}(\bigoplus_{q=s}^{t}(i_!(K^q)[-q]))\lra
\Omega^{\mathscr{D}}_{\alpha}(i_!i^!(X))\oplus I_1
\lra \Omega^{\mathscr{D}}_{\alpha}(\bigoplus_{q=s}^{t}(i_!(Z^q)[-q]))\lra 0.
\end{align}
By Lemma \ref{lem-ex-plus}, the relations (\ref{sub-k}) and (\ref{sub-z}), we have
\begin{align}\label{sub-kz-x}
\Omega^{\mathscr{D}}_{\alpha}(i_!i^!(X))\in [\Omega^{\mathscr{D}}_{\alpha}(\bigoplus_{q=s}^{t}(i_!(N)[-q]))]_{2n+2},  \text{ for all }{\alpha}\ge v.
\end{align}

Let $l:=\ed(C)$ and $L\in C\modcat$ with $C\modcat=[L]_{l+1}$.  Denote $u:=w(j^*(A))$. An analogous argument to the proof of Lemma \ref{lem-adjoint}(3) shows 
\begin{align}\label{omega-u-x}
\Omega^{\mathscr{D}}_q(j_*j^*(X))\in [\Omega^{\mathscr{D}}_q(j_*(L)[-\sup(j^*)])]_{l+1} \text{ for } q\ge u.
\end{align}

Denote $p:=\max\{u,v\}$. 
Applying $\Omega^{\mathscr{D}}_p$ to the triangles (\ref{can-tri}), we have triangles in $\Db{A\modcat}$:
\begin{equation}\label{can-tri-i-bb}
\Omega^{\mathscr{D}}_p(i_{!}i^!(X))\lra
\Omega^{\mathscr{D}}_p(X)\oplus I\lra
\Omega^{\mathscr{D}}_p(j_{*}j^{*}(X))
\lra \Omega^{\mathscr{D}}_p(i_{!}i^!(X))[1]
\end{equation}
for some injective $A$-module $I$ by Lemma \ref{lem-omega-shift}.
Note that all term in the triangles (\ref{can-tri-i-bb}) are $A$-modules. Applying the cohomological functors $H^i:=\Hom_{\Db{A\modcat}}(A,-[i])$, $i=-1,0,1$ 
to the triangles (\ref{can-tri-i-bb}), we have short exact sequences
\begin{equation}\label{can-tri-i-ex}
0\lra
\Omega^{\mathscr{D}}_p(i_{!}i^!(X))\lra
\Omega^{\mathscr{D}}_p(X)\oplus I\lra
\Omega^{\mathscr{D}}_p(j_{*}j^{*}(X))
\lra 0
\end{equation}
By Lemma \ref{lem-ex-plus}, the relations (\ref{sub-kz-x}) and (\ref{omega-u-x}), we have
\begin{align}\label{omega-vu-x-2}
\Omega^{\mathscr{D}}_p(X)\in 
[\Omega^{\mathscr{D}}_{p}(\bigoplus_{q=s}^{t}(i_!(N)[-q]))\oplus \Omega^{\mathscr{D}}_p(j_*(L)[-\sup(j^*)])]_{2n+l+3}.
\end{align}
Due to $X\in A\modcat$, $\Omega^{-p}(X)=\Omega^{\mathscr{D}}_p(X)$.
By Lemma \ref{lem-ext-pi}(2), 
$$X \in 
[\Omega^{p}(\Omega^{\mathscr{D}}_{p}(\bigoplus_{q=s}^{t}(i_!(N)[-q]))\oplus \Omega^{\mathscr{D}}_p(j_*(L)[-\sup(j^*)]))
\oplus\big(\bigoplus_{i=0}^{p-1}\Omega^{i}(D(A))\big)]_{2n+l+p+3}.$$
By Definition \ref{ed-def}, $$\ed(A)\le 2n+l+p+2=2\ed(B)+\ed(C)+\max\{w(i_*(B),w(j^*(A))\}+2.$$

(2)(i) By Lemma \ref{lem-wid}(1), $w(i^!)=w(i_!(B))<\infty$. Then $i^!(X)\in \mathscr{D}^{[\inf(i^!),\sup(i^!)]}(B\modcat)$. Thus there is a canonical triangle
\begin{align}\label{tri-ires-i}
Y[-\sup(i^!)]\cpx{J}\lra i^!(X)\lra \cpx{J}\lra  Y[-\sup(i^!)][1]
\end{align}
for some $Y\in B\modcat$, where $\cpx{J}\in \mathscr{K}^{[\inf(i^!),\sup(i^!)-1]}(\imodcat{B})$.
Applying the functor $i_!$ to the triangle (\ref{tri-ires-i}), we get a new triangle 
\begin{align}\label{tri-ires-2-i}
i_!(Y)[-\sup(i^!)]\lra i_!i^!(X)\lra i_!(\cpx{J})\lra  i_!(Y)[-\sup(i^!)][1].
\end{align}
By Lemma \ref{lem-wid}(2), $i_!(D(B))\in \mathscr{K}^{[-\sup(i^*),-\inf(i^*)]}(\imodcat{A})$. 
Then \begin{align}\label{gj-relation-i}
i_!(\cpx{J})\in i_!(\mathscr{K}^{[\inf(i^!),\sup(i^!)-1]}(\imodcat{B}))\subseteq 
\mathscr{K}^{[\inf(i^!)-\sup(i^*),\sup(i^!)-\inf(i^*)-1]}(\imodcat{B}).
\end{align}
Notice $w(i_!)=w(i^*(A))<\infty$. Then $i_!(Y)\in \mathscr{D}^{[\inf(i_!),\sup(i_!)]}(A\modcat)$ and $$i_!(Y)[-\sup(i^!)] \in \mathscr{D}^{[\inf(i_!)+\sup(i^!),\sup(i_!)+\sup(i^!)]}(A\modcat).$$ 
Since $i^*$ has a left adjoint, it follows from Lemma \ref{lem-wid}(1), $$\sup(i^*(A))=\sup(i^*)\text{ and }w(i^*)=\sup(i^*)-\inf(i^*)\le w(i^*(A)).$$ 
Since $(i^*,i_!)$ is an adjoint pair, $\inf(i_!)=-\sup(i^*(A))$ and $w(i_!)=w(i^*(A))$ by Lemma \ref{lem-wid}(1). Then $$w(i_!)=\sup(i_!)-\inf(i_!)=\sup(i_!)+\sup(i^*(A))=\sup(i_!)+\sup(i^*)=w(i^*(A)).$$
Thus $-\inf(i^*)\le \sup(i_!)$ and $\sup(i^!)-\inf(i^*)\le \sup(i^!)+\sup(i_!)$.

Denote $v:=\sup(i^!)+\sup(i_!)$.
By Lemma \ref{lem-adjoint}(1), $v=w(i_!(A))\ge \sup(i^!)-\inf(i^*)$.
It follows from the relation (\ref{gj-relation-i}) that $\Omega^{\mathscr{D}}_v(i_!(\cpx{J}))=0$. 
By Lemma \ref{lem-omega-coshift}, there is an injective $A$-module $I$ such that 
\begin{align}\label{gfx-sim-gy-i}
\Omega^{\mathscr{D}}_v(i_!i^!(X))\oplus I\simeq \Omega^{\mathscr{D}}_v(i_!(Y)[-\sup(i^!)]).
\end{align}

Let $n:=\ed(B)$ and $N\in B\modcat$ with $B\modcat=[N]_{n+1}$.  An analogous argument to the proof of Lemma \ref{lem-adjoint}(2) shows that
\begin{align}\label{omega-v-x}
\Omega^{\mathscr{D}}_q(i_!i^!(X))\in [\Omega^{\mathscr{D}}_q(i_!(N)[-\sup(i^!)])]_{n+1} \text{ for } q\ge v.
\end{align}

Denote $p:=\max\{u,v\}$. An analogous argument to the proof of Lemma \ref{lem-adjoint}(3) shows that

By Lemma \ref{lem-ex-plus}, the relations (\ref{omega-v-x}) and (\ref{omega-u-x}), we have
\begin{align}\label{omega-vu-x}
\Omega^{\mathscr{D}}_p(X)\in [\Omega^{\mathscr{D}}_p(i_!(N)[-\sup(i^!)])\oplus \Omega^{\mathscr{D}}_p(j_*(L)[-\sup(j^*)])]_{n+l+2}.
\end{align}
Due to $X\in A\modcat$, $\Omega^{-p}(X)=\Omega^{\mathscr{D}}_p(X)$.
By Lemma \ref{lem-ext-pi}(2), $$X \in [\Omega^{p}(\Omega^{\mathscr{D}}_p(i_!(N)[-\sup(i^!)])\oplus \Omega^{\mathscr{D}}_p(j_*(L)[-\sup(j^*)]))\oplus D(A))\oplus\big(\bigoplus_{i=0}^{p-1}\Omega^{i}(D(A))\big)]_{n+l+p+2}.$$
By Definition \ref{ed-def}, $$\ed(A)\le n+l+p+1=\ed(B)+\ed(C)+\max\{w(i_*(B),w(j^*(A))\}+1.$$

(ii) Applying Lemma \ref{lem-adjoint}(4) and (5), we can obtain $$\max\{ \oed(B),\oed(C)\}\le \oed(A).$$

Now, we prove $ \oed(A)\le \oed(B)+\oed(C)+1.$ Indeed, the given recollement (\ref{thm-ext-rec-f}) yields the following canonical triangle
\begin{align}\label{can-tri-left}
j_{!}j^{!}(X)\lra
X\lra i_{*}i^*(X)
\lra j_{!}j^{!}(X)[1]
\end{align}
in $\D{A}$.

On the one hand,
Suppose $\oed(B)=n$. By the definition of the $\Omega$-extension dimension, there exist $\beta\ge 0$ and $N\in B\modcat$ such that  $\Omega^{ \beta}(B\modcat)\subseteq [N]_{n+1}$.  Suppose $X\in \Omega^{\beta}(A\modcat)$. An analogous argument to the proof of Lemma \ref{lem-adjoint}(5) shows that
\begin{align}\label{iv-b-n-fin}
\Omega_{\mathscr{D}}^p(i_*i^*(X))\in [ \Omega_{\mathscr{D}}^p(i_*( N[w(i^*(A))-\sup(i^*)]))]_{n+1}\text{ for any } p\ge w(i_*(B)).
\end{align}

On the other hand, 
Suppose $\oed(C)=l$. By the definition of the $\Omega$-extension dimension, there exist $\gamma \ge 0$ and $L\in B\modcat$ such that  $\Omega^{ \gamma}(C\modcat)\subseteq [L]_{l+1}$.  An analogous argument to the proof of Lemma \ref{lem-adjoint}(4) shows that
\begin{align}\label{iv-b-n-fin-jw}
\Omega_{\mathscr{D}}^p(j_!j^!(X))\in [ \Omega_{\mathscr{D}}^p(j_!(L[w(j^!(A))-\sup(j^!)])))]_{l+1}\text{ for any } p\ge w(j^!(A)).
\end{align}

Set $u:= \max\{w(j^!(A)),w(i_*(B))\}$. 
By Lemma \ref{lem-omega-shift}, applying $\Omega_{\mathscr{D}}^u$ to the triangle (\ref{can-tri-left}), we have a triangle in $\Db{A\modcat}$:
\begin{align}\label{can-tri-left-om}
\Omega_{\mathscr{D}}^u(j_{!}j^{!}(X))\lra
\Omega_{\mathscr{D}}^u(X)\oplus P_1\lra 
\Omega_{\mathscr{D}}^u(i_{*}i^*(X))
\lra \Omega_{\mathscr{D}}^u( j_{!}j^{!}(X))[1]
\end{align}
with $P_1\in \pmodcat{A}$.
Note that all term in the triangles (\ref{can-tri-left-om}) are $A$-modules. Applying the cohomological functors $H^i:=\Hom_{\Db{B\modcat}}(B,-[i])$, $i=-1,0,1$ 
to the triangles (\ref{can-tri-left-om}), we have a short exact sequence
$$0\lra \Omega_{\mathscr{D}}^u(j_{!}j^{!}(X))\lra
\Omega_{\mathscr{D}}^u(X)\oplus P_1\lra 
\Omega_{\mathscr{D}}^u(i_{*}i^*(X))\lra 0.$$
Since the operation $\bullet$ is associative, $\Omega_{\mathscr{D}}^u(X)\in [ \Omega_{\mathscr{D}}^u(j_{!}j^{!}(X))]_1\bullet[\Omega_{\mathscr{D}}^u(i_{*}i^*(X))]_1$.

Set $\alpha:=\max\{\beta,\gamma\}$.
Suppose $X\in \Omega^{\alpha}(A\modcat)$. By the relations (\ref{iv-b-n-fin}) and (\ref{iv-b-n-fin-jw}), we get
\begin{align*}
\Omega_{\mathscr{D}}^u(X)&\in [ \Omega_{\mathscr{D}}^u(j_{!}j^{!}(X))]_1\bullet[\Omega_{\mathscr{D}}^u(i_{*}i^*(X))]_1\\
&\subseteq [ \Omega_{\mathscr{D}}^u(j_!(L[w(j^!(A))-\sup(j^!)])))]_{l+1} \bullet [ \Omega_{\mathscr{D}}^u(i_*( N[w(i^*(A))-\sup(i^*)]))]_{n+1}\\
&\subseteq [ \Omega_{\mathscr{D}}^u(j_!(L[w(j^!(A))-\sup(j^!)])))\oplus  \Omega_{\mathscr{D}}^u(i_*( N[w(i^*(A))-\sup(i^*)]))]_{l+n+2}
\end{align*}
By the definition of the $\Omega$-extension dimension, we get $\oed(A)\le n+l+1=\oed(B)+\oed(C)+1.$

(iii) An analogous argument to the proof of (2)(ii) shows this assertion.

(iv) By (2)(ii), we obtain $\oed(B)\le \oed(A)$. We have to show $\oed(B)\ge \oed(A)$. Indeed, by Lemma \ref{lem-wid}(1), $\mid w(j^!)\mid<\infty$. Then 
$j^!(A\modcat)\subseteq \mathscr{D}^{[\inf(j^!),\sup(j^!)]}(C\modcat).$
Due to $s:=\gd(C)<\infty$, $$j^!(A\modcat)\subseteq \mathscr{K}^{[\inf(j^!)-s,\sup(j^!)]}(\pmodcat{C}).$$
By Lemma \ref{lem-wid}(1), $j_!(C)\in \mathscr{K}^{[\sup(j_!)-w(j_!(C)),\sup(j_!)]}(\pmodcat{A})$. Then
$$j_!j^!(A\modcat)\subseteq \mathscr{K}^{[\sup(j_!)-w(j_!(C))+\inf(j^!)-s,\sup(j_!)+\sup(j^!)]}(\pmodcat{A}).$$
By Lemma \ref{lem-adjoint}(1) and Lemma \ref{lem-wid}(1), $\sup(j_!)+\sup(j^!)=w(j_!(C))$ and $w(j^!)=w(j_!(C))$. Then
$$j_!j^!(A\modcat)\subseteq \mathscr{K}^{[-w(j_!(C))-s,w(j_!(C))]}(\pmodcat{A}).$$
Thus $\Omega_{\mathscr{D}}^p(j_!j^!(A\modcat))=0$ for $p\ge  w(j_!(C))+s+1$. By Lemma \ref{lem-omega-shift} and the triangle (\ref{can-tri-left}), there is a projective $A$-module $P_2$ such that 
\begin{align}\label{gl-main-1}
\Omega_{\mathscr{D}}^p(X)\oplus P_2\simeq  \Omega_{\mathscr{D}}^p(i_{*}i^*(X))\text{ for any }p\ge  w(j_!(C))+s+1.
\end{align}
Set $\delta:=\max\{w(j_!(C))+s+1, w(i_*(B))\}$. By the relations (\ref{iv-b-n-fin}) and (\ref{gl-main-1}), we get
\begin{align*}
\Omega_{\mathscr{D}}^{\delta}(X) \in [ \Omega_{\mathscr{D}}^{\delta}(i_*( N[w(i^*(A))-\sup(i^*)]))]_{n+1}.
\end{align*}
By the definition of the $\Omega$-extension dimension, we get $\oed(A)\le n=\oed(B).$ Thus $\oed(A)=\oed(B).$
Similarly, we can get $$\ed(\Omega^{\infty}(B\modcat))=\ed(\Omega^{\infty}(A\modcat)).$$

(v) An analogous argument to the proof of (2)(iv) shows this assertion.
\end{proof}

\begin{Koro}\label{cor-homo-ideal}
Let $A$ be an Artin algebra and $e \in A$ be an idempotent element of $A$.
Suppose that the canonical surjection $A\ra A/AeA$ is homological.

{\rm (1)} If $\pd(_{A}AeA)<\infty$, then 

{\rm (i)} $\ed(A/AeA)\le \ed(A)+\pd(_{A}AeA)+1$.

{\rm (ii)} $\ed(eAe)\le \ed(A)+\pd(_{A}AeA)$.

{\rm (iii)} $ \ed(A)\le 2\ed(A/AeA)+\ed(eAe)+\pd(_{A}AeA)+3.$

{\rm (2)} If $\pd(_{A}AeA)<\infty$ and $\pd(Ae_{eAe})<\infty$ , then 

{\rm (i)} $\ed(A)\le \ed(A/AeA)+\ed(eAe)+\pd(_{A}AeA)+2.$

{\rm (ii)} $\max\{ \oed(A/AeA),\oed(eAe)\}\le \oed(A)\le \oed(A/AeA)+\oed(eAe)+1.$

{\rm (iii)} 
$\max\{ \ed(\Omega^{\infty}(A/AeA\modcat)),\ed(\Omega^{\infty}(eAe\modcat))\}\le \ed(\Omega^{\infty}(A\modcat)),\text{ and}$
$$ \ed(\Omega^{\infty}(A\modcat))\le \ed(\Omega^{\infty}(A/AeA\modcat))+\ed(\Omega^{\infty}(eAe\modcat))+1.$$

{\rm (iv)} If, in addition, $\gd(eAe)<\infty$, then $$\oed(A/AeA)=\oed(A)\text{ and } \ed(\Omega^{\infty}(A/AeA\modcat))=\ed(\Omega^{\infty}(A\modcat)).$$

{\rm (v)} If, in addition, $\gd(A/AeA)<\infty$, then $$\oed(eAe)=\oed(A)\text{ and } \ed(\Omega^{\infty}(eAe\modcat))=\ed(\Omega^{\infty}(A\modcat)).$$
\end{Koro}

\begin{proof}
From \cite{CPS}, there is a recollement
\begin{align}\label{rec-strat}
\xymatrixcolsep{4pc}\xymatrix{
\D{A/AeA\Modcat} \ar[r]|{i_*=i_!} &\D{A\Modcat} \ar@<-2ex>[l]|{i^*} \ar@<2ex>[l]|{i^!} \ar[r]|{j^!=j^*}  &\D{eAe\Modcat}, \ar@<-2ex>[l]|{j_!} \ar@<2ex>[l]|{j_{*}}
}
\end{align}
where $i_{*}=A/AeA\otimes_{A/AeA}-$, $j_{!}=Ae\otimes_{eAe}^{\mathbb{L}}-$ and $j^!=eA\otimes_{A}^{\mathbb{L}}-$. 

(1) Suppose that $\pd(_AAeA)<\infty$. Then $\pd(_AA/AeA)<\infty$. By Lemma \ref{lem-rec-prop}(1), $i_{*}$ restricts to $\Kb{\rm proj}$ and the recollement (\ref{rec-strat}) extends one step downwards. Note that $w(i_*(A/AeA))=\pd(_AA/AeA)=\pd(_AAeA)+1$. %\cite[Lemma 2.3]{xu14}
It follows from \cite[Corollary 3.2]{apt92} that $\pd(_{eAe}eA)=\pd(_AAeA)$. Then
$w(j^!(A))=\pd(_{eAe}eA)=\pd(_AAeA)$. Now, Corollary \ref{cor-homo-ideal}(1) follows from Theorem \ref{thm-ext-rec}(1).

(2) Suppose $\pd(Ae_{eAe})<\infty$. Then $\mathbb{R}\Hom_{\D{{eAe\text{-Mod}}}}(\mathbb{R}\Hom_{\D{eAe^{\opp}\Modcat}}(Ae,eAe),eAe)\simeq Ae$ in $\D{eAe\mbox{-mod}}$. Thus $j_{!}=Ae\otimes_{eAe}^{\mathbb{L}}-$ has the left adjoint $\mathbb{R}\Hom_{\D{eAe^{\opp}\Modcat}}(Ae,eAe)\otimes_{A}^{\mathbb{L}}-$. By Lemma \ref{lem-rec-prop}(2), the recollement (\ref{rec-strat}) extends one step upwards. Then Corollary \ref{cor-homo-ideal}(2) follows from Corollary \ref{cor-homo-ideal}(1) and Theorem \ref{thm-ext-rec}(2).
\end{proof}

\begin{Koro}\label{cor-tri-ex}
Let
$A=\left(\begin{smallmatrix}
B&M\\
0&C
\end{smallmatrix}\right)$
be the triangular matrix algebra. The the following holds.

{\rm (1)} $\max\{\ed(B),\ed(C)\}\le \ed(A)\le 2\ed(B)+\ed(C)+2.$

{\rm (2)} If $\pd(M_C)<\infty$, then 

{\rm (i)}  $\max\{\ed(B),\ed(C)\}\le \ed(A)\le \ed(B)+\ed(C)+1.$

{\rm (ii)} $\max\{ \oed(B),\oed(C)\}\le \oed(A)\le \oed(B)+\oed(C)+1.$

{\rm (iii)} 
$\max\{ \ed(\Omega^{\infty}(B\modcat)),\ed(\Omega^{\infty}(C\modcat))\}\le \ed(\Omega^{\infty}(A\modcat)),\text{ and}$
$$ \ed(\Omega^{\infty}(A\modcat))\le \ed(\Omega^{\infty}(B\modcat))+\ed(\Omega^{\infty}(C\modcat))+1.$$

{\rm (iv)} If, in addition, $\gd(C)<\infty$, then $$\oed(B)=\oed(A)\text{ and } \ed(\Omega^{\infty}(B\modcat))=\ed(\Omega^{\infty}(A\modcat)).$$

{\rm (v)} If, in addition, $\gd(B)<\infty$, then $$\oed(C)=\oed(A)\text{ and } \ed(\Omega^{\infty}(C\modcat))=\ed(\Omega^{\infty}(A\modcat)).$$

{\rm (3)} If $\pd(_BM)<\infty$, then

{\rm (i)} $\max\{\ed(B),\ed(C)\}\le \ed(A)\le \ed(B)+\ed(C)+\pd(_{B}M)+2.$

{\rm (ii)} $\max\{ \oed(B),\oed(C)\}\le \oed(A)\le \oed(B)+\oed(C)+1.$

{\rm (iii)} 
$\max\{ \ed(\Omega^{\infty}(B\modcat)),\ed(\Omega^{\infty}(C\modcat))\}\le \ed(\Omega^{\infty}(A\modcat)),\text{ and}$
$$ \ed(\Omega^{\infty}(A\modcat))\le \ed(\Omega^{\infty}(B\modcat))+\ed(\Omega^{\infty}(C\modcat))+1.$$

{\rm (iv)} If, in addition, $\gd(C)<\infty$, then $$\oed(B)=\oed(A)\text{ and } \ed(\Omega^{\infty}(B\modcat))=\ed(\Omega^{\infty}(A\modcat)).$$

{\rm (v)} If, in addition, $\gd(B)<\infty$, then $$\oed(C)=\oed(A)\text{ and } \ed(\Omega^{\infty}(C\modcat))=\ed(\Omega^{\infty}(A\modcat)).$$
\end{Koro}
\begin{proof}
(1) Let $e_1:=\left(\begin{smallmatrix}
1&0\\
0&0
\end{smallmatrix}\right)$ and $e_2:=\left(\begin{smallmatrix}
0&0\\
0&1
\end{smallmatrix}\right)$. It follows from $B=Ae_1$ that the canonical surjection $A\ra A/Ae_2A$ is homological.
Then we have a recollement
\begin{align}\label{rec-tri}
\xymatrixcolsep{4pc}\xymatrix{
\D{B\Modcat} \ar[r]|{i_*=i_!} &\D{A\Modcat} \ar@<-2ex>[l]|{i^*} \ar@<2ex>[l]|{i^!} \ar[r]|{j^!=j^*}  &\D{C\Modcat}, \ar@<-2ex>[l]|{j_!} \ar@<2ex>[l]|{j_{*}}
}
\end{align}
where $i_{*}=Ae_1\otimes_{B}^{\mathbb{L}}-$, $i^{!}=e_1A\otimes_{A}^{\mathbb{L}}-$, $j_{!}=Ae_2\otimes_{C}^{\mathbb{L}}-$ and $j^!=e_2A\otimes_{A}^{\mathbb{L}}-$. Note that $i_{*}(B)=Ae_1$. Since $Ae_1$ is a projective $A$-module,  $i_{*}$ restricts to $\Kb{\rm proj}$ and the recollement (\ref{rec-strat}) extends one step downwards by Lemma \ref{lem-rec-prop}(1). Note that $w(i_{*}(B))=\pd_{A}(B)=0$ and $w(j^!(A))=\pd_{C}(e_2A)=\pd_{C}(C)=0$. Then Corollary \ref{cor-tri-ex}(1) follows from Theorem \ref{thm-ext-rec}(2).

(2) Suppose $\pd(M_C)<\infty$. Then $j_{!}(A)=Ae_2\simeq C\oplus M$ as right $C$-modules. Thus $\pd_{C^{\opp}}(Ae_2)<\infty$. Note that $\mathbb{R}\Hom_{\D{C\Modcat}}(\mathbb{R}\Hom_{\D{C^{\opp}\Modcat}}(Ae_2,C),C)\simeq Ae_2$ in $\D{C\Modcat}$. Thus $j_{!}=Ae_2\otimes_{C}^{\mathbb{L}}-$ has the left adjoint $\mathbb{R}\Hom_{\D{C^{\opp}\Modcat}}(Ae_2,C)\otimes_{A}^{\mathbb{L}}-$. By Lemma \ref{lem-rec-prop}(2), the recollement (\ref{rec-strat}) extends one step upwards. Then Corollary \ref{cor-tri-ex}(2) follows from Corollary \ref{cor-tri-ex}(1) and Theorem \ref{thm-ext-rec}(2).

(3) Suppose $\pd(_BM)<\infty$. It follows from $i^{!}(A)=e_1A\simeq B\oplus M$ as $B$-modules that $w(i^{!}(A))=\pd(_Bi^{!}(A))=\pd(_BM)<\infty$. By Lemma \ref{lem-rec-prop}(1), $i^{!}$ restricts to $\Kb{\rm proj}$ and the recollement (\ref{rec-strat}) extends two step downwards. It follows from
$Ae_1A\simeq \left(\begin{smallmatrix}
B\\
0
\end{smallmatrix}\right)
\oplus 
\left(\begin{smallmatrix}
M\\
0
\end{smallmatrix}\right)
$
as $A$-modules that $\pd_{A}(Ae_1A)=\pd_{B}(M)$. Due to $C=A/Ae_1A$, $\pd_{A}(C)=\pd_{B}(M)+1$. Also $j_*(C)\simeq C$ as $A$-modules. Thus $w(j_*(C))=\pd_{B}(M)+1$.
Now, Corollary \ref{cor-tri-ex}(3) follows from Corollary \ref{cor-tri-ex}(1) and Theorem \ref{thm-ext-rec}(2).
\end{proof}

Now, we apply Theorem \ref{thm-ext-rec}(1) to exact context.

Let $R$, $S$ and $T$ be Artin algebras, and let $\lambda:R\to S$ and $\mu:R\to T$ be algebra homomorphisms.
Suppose that $M$ is an $S$-$T$-bimodule together with an element $m\in M$. We say that the quadruple $(\lambda, \mu, M, m)$ is an \emph{exact context} if the following sequence
$$
0\lra R\lraf{(\lambda,\,\mu)}S\oplus T
\lraf{\left({\cdot\,m\,\atop{-m\,\cdot}}\right)}M\lra 0
$$
is an exact sequence of abelian groups, where $\cdot m$ and $m \cdot$ denote the right and left multiplication by $m$ maps, respectively. 

Given an exact context $(\lambda,\mu, M, m)$, there is defined an algebra with identity in \cite{cx19}, called the \emph{noncommutative tensor product} of $(\lambda,\mu, M, m)$ and denoted by  $T\boxtimes_RS$ if the meaning of the exact context is clear. This notion plays a key role in describing the left parts of recollements induced from homological exact contexts (see \cite[Theorem 1.1]{cx19}).

\begin{Theo} \label{hom-dim}
Let $(\lambda, \mu, M, m)$ be an exact context with the noncommutative tensor product $T\boxtimes_RS$. 
Suppose that ${\rm Tor}^R_i(T,S)=0$ for all $i\ge 1$. If $\pd(_RS)<\infty$, then the following hold true:

{\rm (1)} $\ed(T\boxtimes_RS)\le 2\ed(S)+\ed(T)+2+\max\{2,\pd(_RS)+1\}$.

{\rm (2)} $\ed(R)\le 2\ed(S)+\ed(T)+2+\max\{\pd(_RS),1\}$.

{\rm (3)} $ \ed(\left(\begin{smallmatrix}
S&M\\
0&T
\end{smallmatrix}\right))\le 2\ed(T\boxtimes_RS)+\ed(R)+\max\{2,\pd(_RS)+1\}+2.$
\end{Theo}
\begin{proof}
For the exact context $(\lambda,\mu, M, m),$ we have the following two algebra homomorphisms
$$
\rho: S\to T\boxtimes_RS,\; s\mapsto 1\otimes s\;\;\mbox{for}\;\;s\in S,\;\;\mbox{and}\;\;
\phi: T\to T\boxtimes_RS,\; t\mapsto t\otimes 1\;\;\mbox{for}\;\; t\in
T.
$$
Note that $T\boxtimes_RS$ has $T\otimes_RS$ as its abelian groups, while its multiplication is different from the usual tensor product (see \cite{cx19} for details). Let 
$A:=\left(\begin{smallmatrix}
S&M\\
0&T
\end{smallmatrix}\right)$, $B:=M_2(T\boxtimes_RS)$ and
$ \theta:=\left(\begin{smallmatrix}
\rho & \beta\\
0 & \phi
\end{smallmatrix}\right)
: \; A \ra B, $ where $\beta: M\ra T\otimes_RS$ is the unique $R$-$R$-bimodule homomorphism such that $\phi=(m\cdot)\beta$ and $\rho=(\cdot m)\beta$.

Let $
\varphi:\left(\begin{smallmatrix}
S \\
0 
\end{smallmatrix}\right)
\ra 
\left(\begin{smallmatrix}
M \\
T
\end{smallmatrix}\right),\;
\left(\begin{smallmatrix}
s \\
0 
\end{smallmatrix}\right)
\mapsto
\left(\begin{smallmatrix}
sm\\
0 
\end{smallmatrix}\right)
\;\;\mbox{for}\;\; s\in S.$
Then $\varphi$ is a homomorphism of $A$-$R$-bimodules. Denote by $\cpx{P}$ the mapping cone
of $\varphi$. Then $\cpx{P}\in\Cb{A\otimes_{\mathbb{Z}}R\opp}$ and $_A\cpx{P}\in\Cb{\pmodcat{A}}$.
In particular, $\cpx{P}\in \Kb{\pmodcat{A}}$.

By \cite[Theorem 1.1]{cx19}, if ${\rm Tor}^R_i(T,S)=0$ for all $i\ge 1$, then  there is a recollement of derived categories:
\begin{align}\label{rec-exact}
\xymatrixcolsep{4pc}\xymatrix{
\D{B\Modcat} \ar[r]|{i_*=i_!} &\D{A\Modcat} \ar@<-2ex>[l]|{i^*} \ar@<2ex>[l]|{i^!} \ar[r]|{j^!=j^*}  &\D{R\Modcat}, \ar@<-2ex>[l]|{j_!} \ar@<2ex>[l]|{j_{*}}
}
\end{align}
\noindent where $j_!:={_A}\cpx{P}\otimes^{\mathbb{L}}_R-$, $j^!:=\cpx{\Hom}_A(\cpx{P},-)$ and $i_*:={_A}B\otimes^{\mathbb{L}}_B-$ is the restriction functor induced from the algebra homomorphism $\theta:A\ra B$. 
By \cite[Corollary 5.8(1)]{cx12-01}, it follows from $\pd(_RS)<\infty$ that $w(i_*(B))=\pd(_AB)\le \max\{2,\pd(_RS)+1\}$. 
Thus $i_*(B)\in\Kb{\pmodcat{A}}$ and
the recollement $(\ref{rec-exact})$ extends one step downwards.
By an easy calculation, we know that
$$j^!(A)=\cpx{\Hom}_A(\cpx{P},A)\simeq (S\oplus {\rm Con(m\cdot)})[-1]$$
in $\D{C\Modcat}$, where ${\rm Con(m\cdot)}$ stands for the two-term complex $0\to T \stackrel{ m\cdot}{\rightarrow} M\to 0$ with $M$ of degree $0$. 
Since the sequence 
$
0\lra R\lraf{(\lambda,\,\mu)}S\oplus T
\lraf{\left({\cdot\,m\,\atop{-m\,\cdot}}\right)}M\lra 0
$ is exact, we have ${\rm Con}(m\cdot )\simeq {\rm Con}(\lambda)$ in $\D{R\Modcat}$, where ${\rm Con}(\lambda)$ stands for the two-term complex $0\to R \stackrel{\lambda}{\rightarrow} S\to 0$ with $R$ of degree $-1$. 
Due to $\pd(_RS)<\infty$, $j^!(A)\in \Kb{\pmodcat{R}}$ with $w(j^!(A))=\max\{\pd(_RS),1\}.$ 

Since $B:=M_2(T\boxtimes_RS)$ is Morita equivalent to $T\boxtimes_RS$, we have $\ed(C)=\ed(T\boxtimes_RS)$.
Since $A$ is a triangular matrix algebra with the algebras $S$ and $T$ in the diagonal, it follows from Corollary \ref{cor-tri-ex}(1) that $\max\{\ed(S),\ed(T)\}\le \ed(A)\le 2\ed(S)+\ed(T)+2$.
Then this assertion follows from Theorem \ref{thm-ext-rec}(1).
\end{proof}

\begin{Koro} \label{hom-dim-ring}
Suppose that $R\subseteq S$ is an extension of Artin algebras. Let $S'$ be the endomorphism algebra of the $R$-module $S/R$. If $_RS$ is projective, then the following hold true:

{\rm (1)} $\ed(S'\boxtimes_RS)\le 2\ed(S)+\ed(S')+4$.

{\rm (2)} $\ed(R)\le 2\ed(S)+\ed(S')+3$.

{\rm (3)} $ \ed(\left(\begin{smallmatrix}
S&\Hom_R(S,S/R)\\
0&S'
\end{smallmatrix}\right))\le 2\ed(S'\boxtimes_RS)+\ed(R)+4.$
\end{Koro}
\begin{proof}
Let $\lambda:R\rightarrow S$ be the inclusion, $\pi:S\rightarrow S/R$ the
canonical surjection and $\lambda':R\rightarrow S'$ the induced map by right multiplication.
Recall from \cite[Section 3]{cx12-01} that the quadruple $\big(\lambda, \lambda', \Hom_R(S,S/R),\pi \big)$ is an exact context and the noncommutative tensor product $S'\boxtimes_RS$ of this exact context is well defined. It follows from $S\in \pmodcat{R}$ that  ${\rm Tor}^R_i(S', S)=0$ for all $i\ge 1$.
Then Corollary \ref{hom-dim-ring} follows from Theorem \ref{hom-dim}.
\end{proof}

Let $R$ be an Artin algebra and   $M$ an $R$-$R$-bimodule. Recall that the \emph{trivial extension} of $R$ by $M$ is an algebra, denoted by $R\ltimes M$, with abelian group $R\oplus M$ and multiplication: $(r,m)(r',m')=(rr',rm'+mr')$ for $r,r'\in R$ and $m,m'\in M$. 

\begin{Koro}\label{cor-exact-tri}
Let $\lambda: R\ra S$ be a ring epimorphism and $M$ an $S$-$S$-bimodule such that ${\rm Tor}^R_i(M,S)$ = $0$ for all $i\ge 1$. If $\pd(_RS)<\infty$, then

{\rm (1)} $\ed(S\ltimes M)\le 2\ed(S)+\ed(R\ltimes M)+2+\max\{2,\pd(_RS)+1\}$.

{\rm (2)} $\ed(R)\le 2\ed(S)+\ed(R\ltimes M)+2+\max\{\pd(_RS),1\}$.

{\rm (3)} $ \ed(\left(\begin{smallmatrix}
S&S\ltimes M\\
0&R\ltimes M
\end{smallmatrix}\right))\le 2\ed(S\ltimes M)+\ed(R)+\max\{2,\pd(_RS)+1\}+2.$
\end{Koro}
\begin{proof}
Let $T:=R\ltimes M$ and $\mu: R\rightarrow T$ be
the inclusion from $R$ into $T$. By the proof of \cite[Corollary 3.19]{cx17}, $(\lambda, \mu, S\ltimes M, 1 )$ is an exact context with ${\rm Tor}^R_i(T,S)$ = $0$ for all $i\ge 1$ and $T\boxtimes_RS\simeq S\ltimes M$ as algebras. Due to $\pd(_RS)<\infty$, we get that Corollary \ref{cor-exact-tri} follows immediately from Theorem \ref{hom-dim}.
\end{proof}

\begin{Koro}\label{cor-pullback square}
Let $R$ be an Artin algebra, and $I_1$, $I_2$ be ideals of $R$ such that $I_1\cap I_2=0$. Assume $\pd(_RR/I_1)\le \infty$, then

{\rm (1)} $\ed(R/(I_1+I_2))\le 2\ed(R/I_1)+\ed(R/I_2)+2+\max\{2,\pd(_RR/I_1)+1\}$.

{\rm (2)} $\ed(R)\le 2\ed(R/I_1)+\ed(R/I_2)+2+\max\{\pd(_RR/I_1),1\}$.

{\rm (3)} $ \ed(\left(\begin{smallmatrix}
R/I_1&R/(I_1+I_2)\\
0&R/I_2
\end{smallmatrix}\right))\le 2\ed(R/(I_1+I_2))+\ed(R)+\max\{2,\pd(_RR/I_1)+1\}+2.$

\end{Koro}
\begin{proof}
We define $\lambda:R\to R/I_1$ and $\mu:R\to R/I_2$ to be the canonical surjective algebra homomorphisms. By the proof of \cite[ Corollary 3.20]{cx17}, we know that $(\lambda,\mu,R/(I_1+I_2),1)$ is an exact context with ${\rm Tor}^R_i(R/I_2,R/I_1)$ = $0$ for all $i\ge 1$  and $R/I_2\boxtimes_RR/I_1\simeq R/(I_1+I_2)$ as algebras. As $\pd(_RR/I_1)\le \infty$, Corollary \ref{cor-pullback square} follows immediately from Theorem \ref{hom-dim}.
\end{proof}

\begin{Koro} \label{cor-exact-homol}
Let $\lambda: R\to S$ be a homological ring epimorphism. 
Suppose that $I$ is an ideal of $R$ such that the image $J'$ of $I$ under $\lambda$ is a left ideal in $S$ and that the restriction of
$\lambda$ to $I$ is injective. 
Let  $J$ be the ideal of $S$ generated by $J'$. Suppose ${\rm Tor}^R_i(R/I,S)=0$ for all $i\ge 1$ and $\pd(_RS)<\infty$. Then

{\rm (1)} $\ed(S/J)\le 2\ed(S)+\ed(R/I)+2+\max\{2,\pd(_RS)+1\}$.

{\rm (2)} $\ed(R)\le 2\ed(S)+\ed(R/I)+2+\max\{\pd(_RS),1\}$.

{\rm (3)} $ \ed(\left(\begin{smallmatrix}
S&S/J'\\
0&R/I
\end{smallmatrix}\right))\le 2\ed(S/J)+\ed(R)+\max\{2,\pd(_RS)+1\}+2.$
\end{Koro}
\begin{proof}
Let $\mu:R\to R/I$ to be the canonical surjective algebra homomorphism. By the proof of \cite[ Corollary 3.22]{cx17}, we know that $(\lambda,\mu,S/J',1)$ is an exact context and $R/I\boxtimes_RS\simeq S/J$ as algebras. As ${\rm Tor}^R_i(R/I,S)=0$ for all $i\ge 1$ and $\pd(_RS)<\infty$, Corollary \ref{cor-exact-homol} follows immediately from Theorem \ref{hom-dim}.
\end{proof}

\section{Singular equivalences of Morita type with level}\label{sect-t-level}

In this section, we shall prove that the Igusa-Todorov distance is an invariant under singular equivalence of Morita type with level. To begin, we recall some fundamental results concerning singular equivalences of Morita type with level, as detailed in reference \cite{Wang15}.

In this section, suppose that $A$ is a finite dimensional $k$-algebra over a field $k$. Let
$A^e = A\otimes _k A^{\opp}$ be the enveloping algebra of $A$. We identify
$A$-$A$-bimodules with left $A^e$-modules. Denote by $\Omega _{A^e}(-)$ the syzygy functor on the
stable category $A^e\stmodcat$ of $A$-$A$-bimodules. The following terminology is
due to Wang \cite{Wang15}.

\begin{Def}\label{def-level}
{\rm Let $_AM_B$ and $_BN_A$ be an $A$-$B$-bimodule and a $B$-$A$-bimodule,
respectively, and let $l \geq 0$. We say $(M,N)$ defines a {\it singular equivalence
of Morita type with level} $l$, provided that the following conditions are satisfied:

(1) The four one-sided modules $_AM$, $M_B$, $_BN$ and $N_A$ are all finitely generated projective.

(2) There are isomorphisms $M\otimes _B N \simeq \Omega _{A^e}^l(A)$ and
$N\otimes _A M\simeq \Omega _{B^e}^l(B)$ in
$A^e\stmodcat$ and $B^e\stmodcat$, respectively.}
\end{Def}

\begin{Lem}\label{lem-exact}
Let $A$ and $B$ be two finite dimensional algebras, and $F:A\modcat\ra B\modcat$ be an exact functor. If $F$ preserves projective modules, then for each $A$-module $X$ and each non-positive integer $n$, we have $F(\Omega_A^n(X))\simeq \Omega^n_B(F(X))\oplus Q$ in $B\modcat$ for some projective $B$-module $Q$.
\end{Lem}

\begin{Theo}\label{thm-level-ed}
Let $A$ and $B$ be two finite dimensional algebras.
If $A$ and $B$ are singularly equivalent of Morita type with level, then

{\rm (1)} $\mid\ed(A)-\ed(B)\mid\le l.$

{\rm (2)} $\oed(A)=\oed(B)$.

{\rm (3)} $\ed(\Omega^{\infty}(A\modcat))=\ed(\Omega^{\infty}(B\modcat)).$
\end{Theo}
\begin{proof}
(1) We prove $\ed(B)\le l+\ed(A)$. Indeed, suppose $\ed(A)=u$ with $A\modcat=[U]_{u+1}$ for some $U\in A\modcat$ and $u\ge 0$. For each $Y\in B\modcat$, $M\otimes_BY\in A\modcat$. Then $_AM\otimes_BY\in [U]_{u+1}$ and there are short exact sequences
\begin{equation}\label{exact-my}
\begin{cases}
\xymatrix@C=1.5em@R=0.1em{
0\ar[r]& U_0\ar[r] & (M\otimes_BY)\oplus X_{0}'\ar[r]& X_{1}\ar[r]&0,\\
0\ar[r]& U_1\ar[r] & X_{1}\oplus X_{1}'\ar[r]& X_{2}\ar[r]&0,\\
&&\vdots&&\\
0\ar[r]& U_{n-1}\ar[r] & X_{n-1}\oplus X_{n-1}'\ar[r]& X_{n}\ar[r]&0.
}
\end{cases}
\end{equation}
in $A\modcat$ for some $A$-modules $X_i'$ such that $U_i\in [U]_1$ and $X_{i+1}\in [U]_{u-i}$ for each $0\le i\le u-1$.
Since $N_A$ is a projective right $A$-module, $N\otimes_A-:A\modcat\ra B\modcat$ is an exact functor. Applying the exact functor $N\otimes_A-$ to  the exact sequences (\ref{exact-my}), we get exact sequences
\begin{equation}\label{exact-nmy}
\begin{cases}
\xymatrix@C=1.5em@R=0.1em{
0\ar[r]& N\otimes_AU_0\ar[r] & (N\otimes_AM\otimes_BY)\oplus N\otimes_AX_{0}'\ar[r]& N\otimes_AX_{1}\ar[r]&0,\\
0\ar[r]& N\otimes_AU_1\ar[r] & N\otimes_AX_{1}\oplus N\otimes_AX_{1}'\ar[r]& N\otimes_AX_{2}\ar[r]&0,\\
&&\vdots&&\\
0\ar[r]& N\otimes_AU_{n-1}\ar[r] & N\otimes_AX_{n-1}\oplus N\otimes_AX_{n-1}'\ar[r]& N\otimes_AX_{n}\ar[r]&0.
}
\end{cases}
\end{equation}
By Definition \ref{def-level}, we have $N\otimes_AM\simeq \Omega^l_{B^e}(B)$ in $\stmodcat{B^e}$. By \cite[Theorem 2.2]{h60}, there are projective $B^e$-modules $P$ and $Q$ such that
$N\otimes_AM\oplus P\simeq \Omega^l_{B^e}(B)\oplus Q$ as $B^e$-modules.
Then $$N\otimes_AM\otimes_BY\oplus P\otimes_BY\simeq \Omega^l_{B^e}(B)\otimes_BY\oplus Q\otimes_BY$$ as $B$-modules. 
Due to $P,Q\in \pmodcat{B^e}$, $P\otimes_BY$ and $Q\otimes_BY$ are projective $B$-modules. Then
\begin{align}\label{iso-nmy}
N\otimes_AM \otimes_BY \simeq
\Omega^l_{B^e}(B)\otimes_BY
\end{align} in $\stmodcat{B}$.
Note that we have the following exact sequence
\begin{align}\label{be-seq}
0\lra \Omega^{l}_{B^e}(B)\lra R_{l-1}\lra\cdots \lra R_1\lra R_0\lra B\lra 0
\end{align}
in $B^e\modcat$ with projective $B^e$-modules $R_j$ for $0\le j\le l-1$.
It follows from $B_B\in \pmodcat{B^{\opp}}$ that the exact sequence (\ref{be-seq}) splits in $B^{\opp}\modcat$. Then we have the following exact sequence
\begin{align}\label{y-seq}
0\lra \Omega^{l}_{B^e}(B)\otimes_BY\lra R_{l-1}\otimes_BY\lra\cdots \lra R_1\otimes_BY\lra R_0\otimes_BY\lra Y\lra 0
\end{align}
in $B\modcat$. Note that $R_j\otimes_BY\in \pmodcat{B}$ for each $0\le j\le l-1$. Then $\Omega^{l}_{B^e}(B)\otimes_BY\simeq \Omega^{l}_B(Y)$ in $B\stmodcat$. By (\ref{iso-nmy}), we obtain that
\begin{align}\label{iso2-nmy}
N\otimes_AM \otimes_BY \simeq
\Omega^{l}_B(Y)
\end{align} in $\stmodcat{B}$. 
By \cite[Theorem 2.2]{h60}, there are projective $B$-modules $Q_1$ and $Q_2$ such that
\begin{align}\label{iso4-nmy}
(N\otimes_AM \otimes_BY)\oplus Q_1 \simeq
\Omega^{l}_B(Y)\oplus Q_2
\end{align} in $B\modcat$. 
By the exact sequences (\ref{exact-nmy}), we have 
\begin{align}\label{y-in-nu}
\Omega^{l}_B(Y)\in [N\otimes_AU\oplus A]_{u+1}.
\end{align}
By Lemma \ref{lem-ext-pi}(3), 
$$Y\in [\Omega^{-l}(N\otimes_AU\oplus A)\oplus 
\big(\bigoplus_{i=0}^{l-1}\Omega^{-i}(A)\big)]_{l+u+1}.$$
Thus $\ed(B)\le l+u=l+\ed(A)$. Similarly, we can obtain $\ed(A)\le l+\ed(B)$. Hence $\mid \ed(A)-\ed(B)\mid\le l$.

(2) We prove $\oed(B)\le \oed(A)$. Suppose $\oed(A)=u$ with $\Omega^{\alpha}(A\modcat)\subseteq [U]_{u+1}$ for some $U\in A\modcat$ and $\alpha,u\ge 0$. Since $M_B$ and $_AM$ is projective, $M\otimes_B-:B\modcat\ra A\modcat$ is an exact functor which preseves projective modules. Then for each $Y\in \Omega^{\alpha}(B\modcat)$, by Lemma \ref{lem-exact}, we have $M\otimes_BY\in  \Omega^{\alpha}(A\modcat)$. Then $_AM\otimes_BY\in [U]_{u+1}$. 
An analogous argument to the proof of (1) shows that $\Omega^{l}_B(Y)\in [N\otimes_AU\oplus A]_{u+1}$. Thus $\oed(B)\le u=\oed(A)$. Similarly, we can obtain $\ed(A)\le \ed(B)$. Hence $ \ed(A)=\ed(B)$.

(3)  Note that $M\otimes_B(\Omega^{\infty}(B\modcat))\subseteq \Omega^{\infty}(A\modcat)$. Then this result can be obtained by an analogous argument to the proof of (2).
\end{proof}

Now, let $A$ be a finite dimensional algebra over a field $k$, and $I\subseteq A$ be a two-sided ideal. Following \cite{px06},
$I$ is a \textit{homological ideal} if the canonical map $A \rightarrow  A/I$ is a homological epimorphism, that is, the naturally induced functor $\Db{A/I\modcat} \ra \Db{A\modcat}$ is fully faithful.

\begin{Koro}\label{s-cor-homol}
Let $A$ be a finite dimensional algebra over a field $k$ and let $I\subseteq A$ be a homological ideal which has finite projective dimension as an $A$-$A$-bimodule. Then 

{\rm (1)} $\mid\ed(A)-\ed(A/I)\mid\le 2\cdot\pd(_{A^{\rm e}}A/I).$

{\rm (2)} $\oed(A)=\oed(A/I)$.

{\rm (3)} $\ed(\Omega^{\infty}(A\modcat))=\ed(\Omega^{\infty}(A/I\modcat)).$
\end{Koro}
\begin{proof}
By \cite[Theorem 3.9]{qin22}, $A$ and $A/I$ are singularly equivalent of Morita type with level $2\cdot\pd(_{A^{\rm e}}A/I)$. Then the corollary follows from Theorem \ref{thm-level-ed}.
\end{proof}

Recall that an extension $B\subseteq A$ of finite dimensional algebras is \textit{bounded} if there exists $p\ge 1$ such that the tensor power $(A/B)^{\otimes _Bp}$ vanishes,$A/B$ has finite projective dimension as a $B$-$B$-bimodule, and ${\rm Tor}_i^B(A/B, (A/B)^{\otimes_B j})=0$ for all $i, j\ge 1$.

\begin{Koro}\label{s-cor-b-ext}
Let $B\subseteq A$ be a bounded extension with $(A/B)^{\otimes _Bp}=0$ for some $p\ge 1$. Then 

{\rm (1)} $\mid\ed(A)-\ed(B)\mid\le 2\cdot\pd(_{B^{\rm e}}A/B)+p-1.$

{\rm (2)} $\oed(A)=\oed(B)$.

{\rm (3)} $\ed(\Omega^{\infty}(A\modcat))=\ed(\Omega^{\infty}(B\modcat)).$
\end{Koro}
\begin{proof}
By \cite[Theorem 0.3]{qxzz24} and \cite[Proposition 3.3]{qin22}, $A$ and $B$ are singularly equivalent of Morita type with level $2\cdot\pd(_{B^{\rm e}}A/B)+p-1$. Then the corollary follows from Theorem \ref{thm-level-ed}.
\end{proof}

\bigskip
\noindent{\bf Acknowledgements.}
Jinbi Zhang was supported by the National Natural Science Foundation of China (No. 12401038).
%Junling Zheng was supported by the National Natural Science Foundation of China (Grant No. 12001508) and the project of Youth Innovation Team of Universities of Shandong Province (Grant No. 2022KJ314).

\noindent{\bf Declaration of interests.} The authors have no conflicts of interest to disclose.

\noindent{\bf Data availability.} No new data were created or analyzed in this study.

\end{document}